\documentclass{amsart}
\usepackage{amsmath}
\usepackage{newtxmath}
\usepackage{tikz-cd}
\usepackage{fix-cm}
\usepackage{hyperref}

\newtheorem{thm}{Theorem}[section]
\newtheorem{defn}[thm]{Definition}
\newtheorem{prop}[thm]{Proposition}
\newtheorem{lem}[thm]{Lemma}
\newtheorem{cor}[thm]{Corollary}
\newtheorem{rmk}[thm]{Remark}

\DeclareMathOperator{\arcch}{arccosh}

\DeclareMathOperator{\interior}{int}
\DeclareMathOperator{\image}{Im}
\DeclareMathOperator{\hess}{Hess}
\DeclareMathOperator{\area}{Area}

\begin{document}

\title[The existence of inversive distance circle packing]
{The existence of inversive distance circle packing on polyhedral surface}
\author{Xiang Zhu}
\address{Department of mathematics, Shanghai University, Shanghai, China, 200444.}
\email{zhux@shu.edu.cn}

\begin{abstract}
    We prove that for any discrete curvature satisfying Gauss-Bonnet formula, there exist a unique up to scaling inversive distance circle packing in the discrete conformal equivalent class, whose polyhedral metric meets the target curvature. We prove it by constructing diffeomorphism between fiber bundles with cell decomposition based on Teichm\"uller spaces, and each discrete conformal equivalent class is a fiber passing through finite cell with respect to triangulations, which means we can do surgery on the discrete Ricci flow by edge flipping using a generalized Ptolemy equation to ensure it converge and never blow up.
\end{abstract}

\maketitle

\section{Introduction}

\subsection{Statement of results}

The \emph{Inversive distance} is the generalization to the cosine of the intersection angle of two circles. The \emph{Inversive distance circle packing} is a polyhedral surface with disjoint circles (or cones precisely) centered at vertices (or cone points). As we know that the conformal mapping on Riemannian surface keeps the infinitesimal circle, given a (geometric) triangulation on this surface, if we adjust the radii of circles while fixing all the inversive distance with respect to edges, then that is an analog or discretization of a conformal map. A \emph{discrete curvature} is the difference of angles between cone and plane.

In author's thesis \cite{zhu2019discrete}, we proved the main theorem.
\begin{thm}
    Given an inversive distance circle packing, for any target discrete curvature satisfying Gauss-Bonnet formula, there exist a unique one up to scaling which is discrete conformal to the initial one.
\end{thm}
Here the definition of discrete conformal only allows radius changing on the weighted Delaunay triangulation. If the weighted Delaunay condition is about to break, we have a generalized Ptolemy equation for the triangulation to do surgery of edge flipping. This theorem will be claimed precisely again after notations in Theorem \ref{main_f}. This paper is part of translation of \cite{zhu2019discrete} which is not written in English.

\subsection{Related work}

Circle packing is a concept in which a pattern of circles, satisfying certain adjacency relationships such as intersection, tangency, or separation, is arranged on a surface. Different papers may have different requirements for circle packing. For example, Koebe propo along the adjacent edges of a triangular mesh in \cite{koebe1936kontaktprobleme}sed a family of circles on a sphere such that they tangent to each other.

In Thurston's lecture notes \cite{thurston1979geometry}, specifically in Chapter 13, a theorem relating circle packing to discrete curvature was presented in order to prove Andreev's theorem. The theorem states that given a triangulation on a surface, with angle values assigned to each edge, if the curvatures at the vertices satisfy certain linear inequalities, then there exists a circle packing on the surface where the centers of circle coincide with the vertices and the angles between the circles connected by edges are prescribed. It is also possible to set the angle values to zero, corresponding to the circles being tangent along each edge. Several papers \cite{chow2003combinatorial} \cite{marden1990on} and \cite{stephenson2005introduction} have subsequently provided different proofs of Thurston's circle packing theorem. Chow and Luo, inspired by the continuous case of Ricci flow, defined the discrete Ricci flow and used it to prove Thurston's theorem. Similar work has also been done in \cite{bobenko2004variational}. Due to the idea of discretize the continuous model, various models of discrete Ricci flow have emerged. Zhang summarized 18 models of discrete Ricci flow and analyzed their differential properties \cite{zhang2014unified}. Two of these models are of particular interest in this paper.

The first model is the vertex scaling model, which corresponds to the discrete Yamabe flow. In this flow, the cross-ratios of the edge lengths of triangles on both sides of each edge are fixed. Bobenko et al. proved the global rigidity using Schl"afli's formula for hyperbolic volumes in \cite{bobenko2015discrete}, while Gu, Luo, Sun, and Wu used the Ptolemy formula to prove the existence of the discrete Yamabe flow by changing the triangulation, that is, for any initial polyhedral metric with triangulation and target discrete curvatures satisfying the Gauss-Bonnet formula, the discrete Yamabe flow must converge to a target polyhedral metric with triangulation that is conformally equivalent but not necessarily isotopy to the initial mesh, achieving the desired target discrete curvatures \cite{gu2018discrete}\cite{gu2018discreteII}. It is difficult to verify whether the target curvatures meet the requirements or not, and it is also hard to deal with cases that do not meet the requirements in engineering applications, then the existence for only one restriction on target curvature is easy to use. 
Sun et al. conducted various numerical experiments to demonstrate the efficiency and robustness of the discrete conformal mapping algorithm induced by the vertex scaling model \cite{sun2015discrete}. The works of Penner \cite{penner1987the} and Rivin \cite{rivin1994euclidean} on decorated Teichm\"uller spaces and their cell decomposition played a crucial role in the proofs. 

The second model is the inversive distance circle packing, which generally requires the inversive distance to be greater than $1$ and is used to measure the relationship between two disjoint circles. In a sense, this model can be seen as a generalization of the previous model, with the vertex scaling model approaching the inversive distance model as the inversive distance tends to infinity. Initially mentioned by Bowers and Stephenson in 1997 \cite{bowers2004uniformizing}, a detailed explanation was later provided by Bowers et al. \cite{bowers2003planar}. Guo proved the local rigidity of inversive distance circle packing \cite{guo2011local}, while Luo proved the global rigidity \cite{luo2011rigidity}. Related work has also been done by Springborn et al. \cite{springborn2008variational} and Ge et al. \cite{ge2017deformation}, \cite{ge2017deformationII} and \cite{ge2019deformation}. Unlike Thurston's circle packing theorem, which provides explicit conditions on discrete curvatures, it is not easy to tell about which target curvatures allow for the existence of inversive distance circle packing. Following the ideas of Gu et al. in solving the vertex scaling model, by changing the triangulation of the triangular mesh and using a generalized Ptolemy formula to change the triangulation, we prove that all restrictions on discrete curvatures, except for the Gauss-Bonnet formula, can be removed, then the inversive distance circle packing exists.

When refining a triangular mesh to approximate a smooth surface, it is also worth studying whether discrete conformal mappings converge to their continuous case. In 1987, Rodin and Sullivan proved that tangential circle packing converge to Riemann mappings \cite{rodin1987convergence}, with the requirement that the vertices of the mesh have degree 6. Stephenson provided numerous illustrations as evidence \cite{stephenson1999approximation}. He and Schramm reproved the convergence theorem \cite{he1996convergence}, which applies to more general graphs and converges up to the first and second derivatives. Gu, Luo, and Wu proved that under certain conditions, the discrete conformal mappings obtained through vertex scaling converge to Riemann mappings \cite{gu2019convergence}. Other convergence result can be found in \cite{luo2021convergence} and \cite{chen2022bowersstephensons}. The later one is on the convergence of inversive distance circle packing.

We notice that there are some related work by Bobenko and Lutz in \cite{bobenko2023decorated} recently.

\subsection{Contribution and Organization of the paper}

To prove the main theorem, we follow the idea of Guo in \cite{guo2011local}, by using a variational principle on an energy function. The energy function is written into the form of an integration. If we integrate it by using Schl\"afli's formula, we see that the energy function is the volume of a sum of some hyperbolic polyhedra in $\mathbb{H}^3$. If we take derivate, we obtain the discrete Ricci flow. The main purpose is to ensure that this flow is convergent and never blow up.

In Section \ref{sec:teich}, we define a useful discriminant of inversive distance, and write down a generalized Ptolemy equality. The Ricci flow changes the metric on the surface, while the triangulation may break. However, triangulation is not essential, and we should focus on the metric. From this view point, we should do a surgery on triangulation when it is about to break the triangle inequality. In fact, we do it earlier, when the triangulation is about to be not Delaunay. The generalized Ptolemy equality, or essentially the transform map of Teichm\"uller space with length coordinate, shows how to do the surgery of edge flipping.

In Section \ref{sec:del}, we write down a weighted Delaunay inequality with variables of inversive distances and radii of circles. This is asked as an open problem in \cite{bobenko2004variational}, and give another decreasing energy to prove the existence of weighted Delaunay triangulation by edge flipping method in \cite{lutz2022canonical} as well. The proof by isotopy method is in \cite{zhu2019discrete}. With this inequality, we show that the triangulation will never break and the flow will never blow up if we keep the triangulation be always weighted Delaunay. Moreover, we show the relationship of weighted Delaunay between Euclidean and hyperbolic cases.

In Section \ref{sec:diffeo}, for two case of piecewise flat metric and hyperbolic metric with geodesic boundaries, we construct two total spaces with base spaces as Teichm\"uller spaces, and proved the local finite cell decomposition of them and the $C^1$ diffeomorphism between them. The cells are defined by weighted Delaunay inequalities of triangulations, on whose boundaries it is very lengthy to take derivate for radius and especially inversive distance variables, thus we prove them by showing codes.

In Section \ref{sec:proof}, we show that the discrete conformal equivalent class of inversive distance circle packing is a fiber lived on Teichm\"uller space of polyhedral metric, which passed only a finite number of cells. On using variational principle to find the critical point which meets the target discrete curvature satisfying Gauss-Bonnet formula, we have to ensure the Ricci potential is $C^2$, even through the fiber bundle of Teichm\"uller space of hyperbolic metric is trivial. That is why the $C^1$ differential structure on fiber bundle of Teichm\"uller space of polyhedral metric proved in Section \ref{sec:diffeo} is essential. Finally, we introduce the discrete Ricci flow. The advantage of our work is that when computation iterations by Newton's method, it definitely convergence. Moreover, we do not need to find when the weighted Delaunay condition break for edge switch which is annoying in practice, and just check and flip edges using generalized Ptolemy equation on every iteration step.

\section{Preliminary and notation}

\subsection{Triangulation}

Suppose $S$ is a closed surface, and given a non-empty finite subset $V \subset S$, then we call the tuple $(S,V)$ a \emph{marked surface}.

In this paper, we use $\Delta$-complex instead of simplicial complex to define a triangulation, which allows the cells gluing itself. Given a marked surface $(S,V)$, a (topological) triangulation $\mathcal{T}$ is a $\Delta$-complex decomposition up to an isotopy fixing $V$, whose set of zero-cells is $V$ \cite{hatcher2002algebraic}.

We only consider the marked surface with the negative Euler characteristic when punched, which is $\chi(S \setminus V)<0$. If not, there is no triangulation for it. Also, the marked surface should be oriented, or we would study its oriented double cover instead. 

We denote $e_{ij}$ to be the one-cell of $\mathcal{T}$, or the \emph{edge}, connecting $v_i$ and $v_j$, and $f_{ijk}$ to be the two-cell of $\mathcal{T}$, or the \emph{face}, surrounded by $e_{ij}$, $e_{ik}$ and $e_{jk}$. Let $n$ be the number of vertices $\left|V\right|$, and $g$ be the genus of $S$, then the number of edges is $\left|E\right|=6g-6+3n$, and the number of faces is $\left|F\right|=4g-4+2n$, where $E=E(\mathcal{T})$ is the set of edges and $F=F(\mathcal{T})$ is the set of faces.	We use $\mathbb{R}^A$ to represent the space of function $\{\,\mathbf{x} \colon A \to \mathbb{R} \mid \mathbf{x}(a_i)=x_i \in \mathbb{R}\,\}$ defined on a finite set domain $A=\{a_1,\dots,a_n\}$. Here $A$ might be $V$ or $E$.

We denote a \emph{hinge} $\Diamond_{ij;kl}$ to be the edge $e_{ij}$ with two faces $f_{ijk}$ and $f_{ijl}$ at both sides. A \emph{flipping} of $e_{ij}$ is to replace the hinge $\Diamond_{ij;kl}$ by $\Diamond_{kl;ij}$. There are infinitely many triangulations on a marked surface in general, however, Hatcher \cite{hatcher1991triangulations}, Mosher\cite{Mosher1988Tiling} and Penner \cite{penner1987the} proved the following theorem.
\begin{thm}
    Let $\mathcal{T}$ and $\mathcal{T}'$ be two triangulations of $(S,V)$, then there exists finite number of triangulations $\mathcal{T}=\mathcal{T}_1,\mathcal{T}_2,\dots,\mathcal{T}_k=\mathcal{T}'$ such that $\mathcal{T}_i$ can flip to $\mathcal{T}_{i+1}$ where $1\le i<k$.    
\end{thm}

Given a marked surface $(S,V)$, for any point $p\in S\setminus V$ there exists a neighborhood isometric to a region in $\mathbb{E}^2$, and for any vertex $v_i \in V$ there exists a neighborhood isometric to a region in a Euclidean cone with a cone angle of $\varphi_i$, then we say that $(S,V)$ has a \emph{polyhedral metric} or a \emph{piecewise flat metric}, denoted by $d_f$, where $f$ means \emph{flat}. A marked surface with a metric $(S,V,d_f)$ is called a \emph{polyhedral surface} or a \emph{piecewise flat surface}. For any points $p,q \in S$, we use $d_f(p,q)$ to indicate the distance between $p$ and $q$. The \emph{discrete curvature} at $v_i$ is $K_i \coloneqq 2\pi-\varphi_i < 2\pi$. The Gauss-Bonnet formula is  $\sum_{i=1}^{n}K_i=2\pi \chi(S)$.

\begin{thm}\label{triangle}		
    Any piecewise flat surface $(S,V,d_f)$ has a \emph{geodesic triangulation}, that is, a $\Delta$-complex decomposition with all edges geodesic on $d_f$ \cite{mitchell1987the}.
\end{thm}

\subsection{Teichm\"uller space}

The \emph{Teichm\"uller space of piecewise flat metric} on $(S,V)$, denoted by $Teich_f(S,V)$, is the space of all piecewise flat metric on $(S,V)$ considered up to isometry isotopic to the identity map fixing $V$. The equivalence class of a piecewise flat metric $d_f$ is denoted by $[d_f]$.

Given a triangulation $\mathcal{T}$ of $(S,V)$, the edge length of $\mathcal{T}$ should be in the convex polytope
\[
    \mathbb{R}^{E(\mathcal{T})}_\Delta \coloneqq
    \left\{\, \mathbf{x} \in \mathbb{R}^{E(\mathcal{T})} \mid
    \forall f_{ijk} \in F(\mathcal{T}),\,
    0<x(e_{ij})<x(e_{jk})+x(e_{ki}) \,\right\}
\]
Given a length function $\mathbf{x} \in \mathbb{R}^{E(\mathcal{T})}_\Delta$, replace $f_{ijk}$ by a Euclidean triangle of edge lengths $\mathbf{x}(e_{ij}),\mathbf{x}(e_{jk}),\mathbf{x}(e_{ki})$, and glue them by isometries along the corresponding edges, then we construct a piecewise flat metric $d_\mathbf{x}$ and produce an injective map
\[
    \Phi_\mathcal{T} \colon \mathbb{R}^{E(\mathcal{T})}_\Delta
    \to Teich_f(S,V) \quad \mathbf{x} \mapsto [d_\mathbf{x}].
\]
Let $P_f(\mathcal{T}) \coloneqq \Phi_\mathcal{T}(\mathbb{R}^{E(\mathcal{T})}_\Delta)$, by Theorem \ref{triangle} we know that
\[
    Teich_f(S,V)=\bigcup_\mathcal{T}P_f(\mathcal{T}), 
\]
of which the union is over all triangulations of $(S,V)$. 

After flipping a convex hinge of $\mathcal{T}$ we get $\mathcal{T}'$, then we know that the transform map $\Phi_\mathcal{T}^{-1} \circ \Phi_\mathcal{T}'$ is real analytic by the cosine law. If we flip over a five-point star, the metric does not change, which is a pentagon relation. Thus, $Teich_f(S,V)$ is a simply connected real analytic $6g-6+3n$ dimension manifold.

Given a surface $\Sigma$ with or without boundary, the \emph{Teichm\"uller space of hyperbolic metric} on $\Sigma$, denoted by $Teich(\Sigma)$, is the space of all hyperbolic metric with closed geodesic boundary on $\Sigma$ considered up to isometry isotopic to the identity map. We denote a hyperbolic metric by $d$ and its isotopy class by $[d]$.

In this paper we only consider the case that $\Sigma$ is oriented, compact with boundaries, and $\chi(\Sigma)<0$. By gluing a topological open dist $D_i$ on every boundary of $\Sigma$, and selecting any point in $D_i$ as the vertex $v_i$, we get a closed marked surface $(S,V)$. The inverse operation is also legal. We call that $(S,V)$ and $\Sigma$ are \emph{related}. A \emph{truncated triangulation} of $\Sigma$ is the intersection of $\Sigma$ and a triangulation of its related surface $(S,V)$, denoted by $\mathcal{T}=(\Gamma,E,F)$, where $\Gamma=\{\,\gamma_i=\partial D_i,\,i=1 \dots n\,\}$ and $n=\left|V\right|$. For any hyperbolic metric $d$ on $\Sigma$, there exist a unique geodesic isotopic to $e_i \in E$ and orthogonal to the geodesic boundary $\partial \Sigma$ \cite{luo2007teichmuller} \cite{dai2008variational}.

\begin{defn}\label{Oemga}	
    Given a hyperbolic surface with $(\Sigma,d)$ and a truncated triangulation $\mathcal{T}$, we define a length function    
    \[
        \mathbf{x} \colon E(\mathcal{T}) \to \mathbb{R}_{>0}
        \quad e_i \mapsto x_i
    \]
    or $\mathbf{x} \in \mathbb{R}_{>0}^{E(\mathcal{T})}$, called the \emph{length coordinate} of $d$, where $x_i$ is the length of the unique geodesic isotopic to $e_i$ and orthogonal to $\partial \Sigma$. We define the coordinate chart with respect to $\mathcal{T}$ by
    \[
        \Omega_\mathcal{T}^{-1} \colon Teich(\Sigma) \to \mathbb{R}_{>0}^{E(\mathcal{T})}
        \quad [d] \mapsto \mathbf{x}.
    \]
\end{defn}

\begin{thm}\label{teich home}		
    $\Omega_\mathcal{T}$ is a homeomorphic map \cite{dai2008variational}\cite{ushijima1999canonical}.
\end{thm}

\subsection{Weighted Voronoi-Delaunay duality}
In the paper \cite{zhu2019discrete}, Zhu proved that weighted Voronoi decomposition on polyhedral surface is a unique CW decomposition by revising the definition of Voronoi cell. Also, he proved that the existence and uniqueness of weighted Delaunay triangulation on a polyhedral surface by using a construction of the isotopy cover map.

Given a piecewise flat surface $(S,V,d_f)$, we define the \emph{weight function} on every vertex to be $\mathbf{w} \colon V \to \mathbb{R},v_i \mapsto w_i$, or $\mathbf{w} \in \mathbb{R}^V$. The original definition of \emph{weighted Voronoi cell} is
\[
    \left\{\,p \in S \mid d_f(p,v_i)^2-w_i \le d_f(p,v_j)^2-w_j,
    \forall j \ne i\,\right\}.
\]

However, in order to ensure that every cell with respect to a vertex exist and be open and simply connected, Zhu defined the \emph{inner weighted Voronoi cell}. In this paper for inversive distance case, the weight function is positive and with restrictions. Since the geometry meaning of the weight is the square of a radius, we use radii instead. Define the radii as $\mathbf{r} \colon V \to \mathbb{R}_{>0},\, v_i \mapsto r_i$ or $\mathbf{r} \in \mathbb{R}_{>0}^V$, the domain of the weight is defined to be
\[
    R = \{\,\mathbf{r} \in \mathbb{R}_{>0}^V \mid
    0<r_i<\mathrm{Inj}(v_i) \mbox{ and }
    \forall i \ne j,r_i+r_j<d_f(v_i,v_j) \,\}.
\]
This domain can be extended while keeping the above propositions, however, for the inversive distance case, it should be restricted like this. This domain is related to the piecewise flat metric $d_f$ on $(S,V)$, but we omit these symbols because they can be recognized from the context. See \cite{zhu2019discrete} for detail.

\begin{defn}	
    Given a piecewise flat surface $(S,V,d_f)$, and a weight function $\mathbf{r} \in R$, the \emph{inner weighted Voronoi cell} of $v_i$, denoted by $Vor_f(v_i)$, is defined to be the set of all $p\in S$ satisfying that 
    \begin{itemize}
        \item there exists a unique shortest geodesic on $S$ connecting $p$ and $v_i$, and
        \item for any $j\neq i$, $d_f(p,v_i)^2-r_i^2 < d_f(p,v_j)^2-r_j^2$.
    \end{itemize}    
\end{defn}

\begin{thm}\label{vor_f}		
    There exists a unique CW decomposition of $S$, called the \emph{weighted Voronoi decomposition}, of which the set of all the 2-cells is $\{\, Vor_f(v_i) \mid v_i \in V \,\}$.
\end{thm}

\begin{defn}\label{del_def}
    Given a piecewise flat surface $(S,V,d_f)$ with weight $\mathbf{w} \in W$, there exist a unique CW decomposition whose 1-cells are geodesics connecting vertices, called \emph{weighted Delaunay tessellation}, to be the dual graph of the weighted Voronoi decomposition.

    The \emph{weighted Delaunay triangulation} is to subdivide polygon faces of the weighted Delaunay tessellation into triangles by connecting some geodesic diagonals in any ways.
\end{defn}

Suppose the corresponding weighted Delaunay triangulation is $\mathcal{T}=(V,E,F)$, for any face $f_{ijk} \in F$, we denote the dual of $f_{ijk}$ on $S$ by $O_{ijk}$, which is a 0-cell of the weighted Voronoi decomposition. By definition, there exist a positive real number $\rho_{ijk}$ satisfied that 
\[
    \begin{aligned}
        d_f(O_{ijk},v_i)^2-\rho_{ijk}^2&=r_i^2\\
        d_f(O_{ijk},v_j)^2-\rho_{ijk}^2&=r_j^2\\
        d_f(O_{ijk},v_k)^2-\rho_{ijk}^2&=r_k^2.
    \end{aligned}
\]

There is a circle, called the \emph{orthogonal circle}, denoted by $\odot O_{ijk}$, centered at $O_{ijk}$ with radius $\rho_{ijk}$, and orthogonal to three circles centered at $v_i,v_j,v_k$ with radius $r_i$, $r_j$ and $r_i$. Note that these notations above also hold when $\mathcal{T}$ is only geodesic but not weighted Delaunay, except that $O_{ijk}$ may not exist on $S$. For this case we just immerse the local part in to $\mathbb{E}^2$ to find $O_{ijk}$.

\begin{defn}\label{loc_del}
    Given a face $f_{ijk}$, we denote the distance between $O_{ijk}$ and edge $e_{ij}$ by $h_{ij,k}$, which is positive when $O_{ijk}$ and $v_k$ are on the same side of $e_{ij}$, and negative when they are on the different sides.
    
    An edge $e_{ij}$ is called to be \emph{local weighted Delaunay} if $h_{ij,k} + h_{ij,l} \ge 0$.
\end{defn}

\begin{thm}\label{local_whole_f}
    Given a  piecewise flat surface $(S,V,d_f)$ with a weight function $\mathbf{r} \in R$, we know that:
    \begin{itemize}
        \item The weighted Delaunay triangulation exists.
        \item The triangulation is weighted Delaunay, if and only if all the edges are local weighted Delaunay \cite{glickenstein2008geometric}\cite{bobenko2007discrete}\cite{gorlina2011weighted}.
        \item The weighted Delaunay triangulation is unique up to finite diagonal switches.
    \end{itemize}
\end{thm}

Zhu also introduced the weighted Voronoi-Delaunay duality on hyperbolic surface with geodesic boundaries in \cite{zhu2019discrete}. The weight function has no limitations, but it must be positive.

\begin{lem}\label{gbh_sinh}
    Let $r_1,r_2>0$ and $\gamma_1,\gamma_2$ be two geodesics that do not intersect in $\mathbb{H}^2$, then the set of all points $q \in \mathbb{H}^2$ satisfied
    \begin{equation}\label{vor_sinh}
        r_1 \sinh d(q,\gamma_1)=r_2 \sinh d(q,\gamma_2),
    \end{equation}
    is a geodesic, which is orthogonal to the geodesic segment connecting $\gamma_1,\gamma_2$. If the equal sign in the equation \eqref{vor_sinh} is changed to the less-than sign, this will be the half space containing $\gamma_1$ partitioned by the geodesic, and if it is changed to the greater-than sign, this will be the half space containing $\gamma_2$.
\end{lem}

\begin{defn}
    Given a hyperbolic surface $(\Sigma,d)$ with geodesic boundaries $\partial \Sigma=\{\gamma_i,\dots,\gamma_n\}$ and a weight function $\mathbf{r} \in \mathbb{R}^V_{>0}$, 
    the \emph{inner weighted Voronoi cell} of $v_i$ in this case, denoted by $Vor(\gamma_i)$, is defined to be the set of all $p \in S$ such that there exists a unique geodesic connecting $p$ and $\gamma_i$ whose length is the distance $d(p,\gamma_i)$, and for any $j \ne i$, the inequality
    \[
        r_i \sinh d(v_i,\gamma_i)<r_j \sinh d(v_j,\gamma_j)
    \]
    holds.
\end{defn}

\begin{thm}\label{gbh_vor}
    Let $(S,V)$ be the related surface of $\Sigma$, namely, $\Sigma = S \setminus \bigcup_{i=1}^n D_i$ where $D_i$ are disjoint open disks containing $v_i$. Given a weight function $\mathbf{r} \in \mathbb{R}^V_{>0}$, there exists a CW decomposition of $S$, with $\{\, Vor(\gamma_i) \cup D_i \mid i=1,\dots,n \,\}$ as 2-cells. This decomposition is unique restricted on $\Sigma$.
\end{thm}

We continue to use the previous name, which is called \emph{weighted Voronoi decomposition} of $\Sigma$ or $(S,V)$ for convenience. 

\begin{thm}\label{gbh_unique}
    There exist a unique geodesic \emph{truncated} triangulation of $(\Sigma,d)$ up to finite diagonal switches with respect to $\mathbf{r} \in \mathbb{R}^V_{>0}$, called the \emph{weighted Delaunay triangulation}, to be the dual graph of the weighted Voronoi decomposition adding some geodesic diagonals into truncated polygons other than truncated triangle.
\end{thm}
 The word \emph{truncated} means if we fill the boundaries with disks, the triangles will be normal topological ones.

\subsection{Inversive distance circle packing}

In the paper \cite{guo2011local}, Guo proved the local rigidity of inversive distance circle packing. We adopt the notations from this paper.
\begin{defn}
    Given two circles with radii $r_1,r_2>0$, and the distance between their centers is $l_{12}$, then their \emph{inversive distance} is defined by
    \begin{equation}\label{inveuc}
        I_{12} \coloneqq \frac{l_{12}^2-r_1^2-r_2^2}{2r_1 r_2}
    \end{equation}
\end{defn}
 
The definition also holds for two circles with centers located at vertices on a piecewise flat surface. If there is a geodesic connecting these two vertices without passing through other vertices, the definition remains valid. For a geodesic that connects the same vertex without passing through other vertices, if the radius of the circle at that vertex is smaller than the injective radius at that point, it can be considered as two disjoint circles located at the two ends of this geodesic. Similarly, the inversive distance can be defined as well.

When $r_1+r_2<l_{12}$, we have $I_{12}>1$, which is the main case discussed in this paper. We do not care the case of $I_{12} \le 1$, because the intersection angle should be studied instead.

Given a piecewise flat surface $(S,V,d_f)$ with a geodesic triangulation $\mathcal{T}=(V,E,F)$, suppose $\mathbf{I} \in \mathbb{R}_{>1}^{E(\mathcal{T})}$ and $\mathbf{r} \in \mathbb{R}_{>0}^V$, or
\[
    \begin{array}{c}
        \mathbf{I} \colon E(\mathcal{T}) \to (1,\infty) \quad
        e_{ij} \mapsto I_{ij} > 1\\
        \mathbf{r} \colon V \to \mathbb{R}_{>0} \quad v_i \mapsto r_i>0
    \end{array}
\]
then we call the collection of every geodesic circles with radius $r_i$ centered at vertex $v_i \in V$ on $S$ as an \emph{inversive distance circle packing} with respect to the inversive distance $\mathbf{I}$, if every couple of circles between the geodesic edge $e_{ij} \in E$ has the inversive distance $I_{ij} \in \mathbf{I}$.

We give the inversive distance circle packing another notation. Given a  marked surface $(S,V)$ with a topological triangulation $\mathcal{T}$, we define the map
\[
    L_f \colon \mathbb{R}_{>1}^{E(\mathcal{T})} \times
    \mathbb{R}_{>0}^V \to \mathbb{R}_{>0}^{E(\mathcal{T})}
    \quad (\mathbf{I},\mathbf{r}) \mapsto \mathbf{l}
\]
where
\[
    \mathbf{l} \colon E(\mathcal{T}) \to \mathbb{R}_{>0} \quad
    e_{ij} \mapsto \sqrt{r_i^2+r_j^2 + 2I_{ij} r_i r_j}.
\]

If $L_f(\mathbf{I},\mathbf{r}) \in \mathbb{R}^{E(\mathcal{T})}_\Delta$, or the image of the map satisfies every triangle inequalities on all faces of $\mathcal{T}$, we take an element of the isotopy class of $\Phi_\mathcal{T} \circ L_f(\mathbf{I},\mathbf{r})$, called the piecewise flat metric $d_f$. Then the triangulation $\mathcal{T}$ becomes a geodesic triangulation of the piecewise flat surface $(S,V,d_f)$. We say that $(\mathbf{I},\mathbf{r})$ is an inversive distance circle packing of $(S,V,d_f)$ with respect to $\mathcal{T}$.

Note that $L_f$ is related to the triangulation, however, we omit the symbol $\mathcal{T}$ on the notation on the mapping, since it appears in the notation of the domain $\mathbb{R}_{>1}^{E(\mathcal{T})} \times \mathbb{R}_{>0}^V$.

At the end of this section, we define the \emph{lifting} of $L_f$. Denote that
\[
    Q_f(\mathcal{T}) \coloneqq L_f^{-1}(\mathbb{R}^{E(\mathcal{T})}_\Delta)
    \subset \mathbb{R}_{>1}^{E(\mathcal{T})} \times \mathbb{R}_{>0}^V.
\]

Define the lifting map
\[
    \tilde{L}_f \colon Q_f(\mathcal{T}) \to 
    \mathbb{R}^{E(\mathcal{T})}_\Delta \times \mathbb{R}_{>0}^V
    \quad (\mathbf{I},\mathbf{r}) \mapsto (L_f(\mathbf{I},\mathbf{r}),\mathbf{r})
\]
and the projection map
\[
    \pi \colon \mathbb{R}_{>0}^{E(\mathcal{T})} \times \mathbb{R}_{>0}^V
    \to \mathbb{R}_{>0}^{E(\mathcal{T})} \quad (\mathbf{l},\mathbf{r}) \mapsto \mathbf{l}
\]
then we have
\[
    \pi \circ \tilde{L}_f = L_f.
\]

\begin{prop}\label{lift_inject}
    $\tilde{L}_f$ is an injective map, then a real analytic homeomorphism from its domain to its image.
\end{prop}

\begin{proof}
    If $\tilde{L}_f(\mathbf{I},\mathbf{r})=\tilde{L}_f(\mathbf{I}',\mathbf{r}')$, we have $(L_f(\mathbf{I},\mathbf{r}),\mathbf{r})=(L_f(\mathbf{I}',\mathbf{r}'),\mathbf{r}')$, then $\mathbf{r}=\mathbf{r}'$. Thus, we know by $\sqrt{r_i^2+r_j^2 + 2I_{ij} r_i r_j}=\sqrt{r_i^2+r_j^2 + 2I'_{ij} r_i r_j}$ that for any $e_{ij} \in E$ and $I_{ij}=I'_{ij}$, or $\mathbf{I}=\mathbf{I}'$.
\end{proof}

\section{Coordinate transformation on Teichm\"uller space}\label{sec:teich}

Suppose $\Sigma = S \setminus \bigcup_{i=1}^n D_i$ is a related surface of a marked surface $(S,V)$ with $n$ vertices, where $D_i$ are distinct open disks containing $v_i$. By Theorem \ref{teich home} we know that the coordinate chart of $Teich(\Sigma)$ can be defined via a triangulation of $(S,V)$. Using the symbol in definition \ref{Oemga} we denote the chart by $\Omega_\mathcal{T} \colon \mathbb{R}_{>0}^{E(\mathcal{T})} \to Teich(\Sigma)$. In this section we solve transform maps between charts with respect to different triangulations.

\begin{defn}
    We define the \emph{discriminant of inversive distance} by
    \[
        \Delta_{abc} \coloneqq a^2+b^2+c^2+2abc-1.
    \]
\end{defn}

\begin{prop}
When $a,b,c>1$, let $x=\arcch a,y=\arcch b,z=\arcch c$, then
    \[
        \Delta_{abc}=4\cosh\frac{x+y+z}2 \cosh\frac{y+z-x}2
        \cosh\frac{x+z-y}2 \cosh\frac{x+y-z}2 >4 .
    \]
\end{prop}

The following theorem tells that how does the inversive distance of its circle packing change when edge flipping occurs.

\begin{thm}\label{teich home formula}
Suppose $\mathcal{T}$ is a triangulation of $(S,V)$, and we get another triangulation $\mathcal{T}'$ by flipping the diagonal of the hinge $\Diamond_{ij;kl}$, then the transform map between these two coordinate charts is
\[
    \Omega_{\mathcal{T}'}^{-1} \circ \Omega_\mathcal{T} \colon
    \mathbb{R}_{>0}^{E(\mathcal{T})} \to \mathbb{R}_{>0}^{E(\mathcal{T}')}
    \quad \mathbf{l} \mapsto \mathbf{l}'
\]
\[			
    \mathbf{l}'(e_m)=			
    \begin{cases}
        \mathbf{l}(e_m) & e_m \ne e_{kl} \\
        \arcch f & e_m = e_{kl}		
    \end{cases}
\]
where $a=\cosh \mathbf{l}(e_{ki}),b=\cosh \mathbf{l}(e_{il}),c=\cosh \mathbf{l}(e_{lj}),d=\cosh \mathbf{l}(e_{jk}),e=\cosh \mathbf{l}(e_{ij})$ and
\begin{equation}\label{f}
    f = \frac{ab + cd + ace + bde + \sqrt {\Delta_{ade}} \sqrt {\Delta _{bce}}}{e^2-1}.
\end{equation}

For any two different triangulation of $(S,V)$, we can transform from one to another by finite steps of flipping, and the corresponding transform map is the composition of the transform maps with respect to the flipping.
\end{thm}

To make the formula clear, here we use $d$ and $e$ without subscript to represent the hyperbolic cosine of the lengths instead of metrics or edges. See Figure \ref{fig:4hyp}.

\begin{figure}[ht]
\centering
\includegraphics{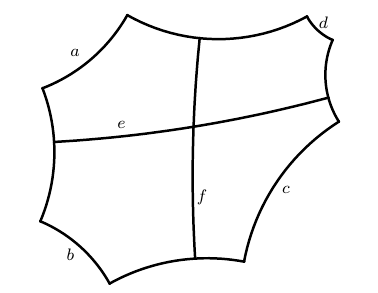}
\caption{Notation in a hyperbolic right angle octagon related with a flipping hinge.}
\label{fig:4hyp}
\end{figure}    

\begin{proof}		
We only need to prove that for any hinge $\Diamond_{ij;kl} \subset \Sigma$ with a given hyperbolic metric, the formula \eqref{f} holds. Denote the length of geodesic segment on $\partial D_i$ between $e_{ik}$ and $e_{ij}$ by $\alpha$ and the length of geodesic segment on $\partial D_i$ between $e_{ij}$ and $e_{il}$ by $\beta$, and apply the formula \eqref{cosh} on $f_{ijk},f_{ijl}$, then we have
\[
    \cosh \alpha = \frac{ae+d}{\sqrt{a^2-1}\sqrt{e^2-1}}\quad
    \cosh \beta = \frac{be+c}{\sqrt{b^2-1}\sqrt{e^2-1}}.
\]

Therefore,
\[
    \begin{aligned}
        \sinh^2 \alpha=&\cosh^2 \alpha-1\\
        =&\frac{(ae+d)^2-(a^2-1)(e^2-1)}{(a^2-1)(e^2-1)}\\
        =&\frac{2ade+d^2+a^2+e^2-1}{(a^2-1)(e^2-1)}\\
        =&\frac{\Delta_{ade}}{(a^2-1)(e^2-1)},\\
        \sinh^2 \beta=&\frac{\Delta_{bce}}{(b^2-1)(e^2-1)}.
    \end{aligned}		
\]

Note that $\alpha,\beta>0$, and the square root on both sides are
\[
    \sinh \alpha = \frac{\sqrt{\Delta_{ade}}}{\sqrt{a^2-1}\sqrt{e^2-1}}\quad
    \sinh \beta = \frac{\sqrt{\Delta_{bce}}}{\sqrt{b^2-1}\sqrt{e^2-1}}
\]

From $\cosh(\alpha+\beta) = \cosh\alpha\cosh\beta+\sinh\alpha\sinh\beta$ and the formula \eqref{cosh} applying on hyperbolic right angle hexagon, we have
\[
    \cosh(\alpha+\beta)=\frac{(ae+d)(be+c)+\sqrt{\Delta_{ade}}\sqrt{\Delta_{bce}}}
    {\sqrt{a^2-1}\sqrt{b^2-1}(e^2-1)}
    =\frac{ab+f}{\sqrt{a^2-1}\sqrt{b^2-1}}.
\]

The formula \eqref{f} holds while both sides times $\sqrt{a^2-1}\sqrt{b^2-1}$ and minus $ab$. For any $a,b,c,d,e>1$, the formula \eqref{f} is an elementary function and an analytic function as well. Thus, the mapping $\Omega_{\mathcal{T}'}^{-1} \circ \Omega_\mathcal{T}$ is real analytic, and homeomorphic as well by Theorem \ref{teich home}.

\end{proof}

\begin{cor}	
Consider the formula \eqref{f} ignoring the meaning of the variables. Then
\begin{enumerate}
\item when $a,b,c,d,e>1$, we have $f>1$. Furthermore, fixing $a,b,c,d>1$ as constants, regarding $f$ as a function of the single variable $e$ like $f=f(e)$, then $f(1+)=+\infty$, $f(+\infty)=1$, and $f$ is strictly monotonically decreasing on $(1,+\infty)$.

\item By calculating directly, we have
    \begin{equation}\label{delta}
        \begin{aligned}
        \sqrt{\Delta _{a b f}}&=\frac{(d+a e) \sqrt{\Delta _{b c e}}+(c+b e) \sqrt{\Delta _{a d e}}}{e^2-1},\\	
        \sqrt{\Delta _{c d f}}&=\frac{(a+d e) \sqrt{\Delta _{b c e}}+(b+c e) \sqrt{\Delta _{a d e}}}{e^2-1}.
        \end{aligned}
    \end{equation}

\item $a,b,c,d,e,f$ satisfy the equation
    \begin{equation}\label{abcdef}
        \begin{aligned}
            a^2 + b^2 + c^2 + d^2 + e^2 + f&^2\\
            + 2 (ade + bce + abf + cdf + abcd + acef + bdef&)\\
            - a^2 c^2 - b^2 d^2 - e^2 f^2 - 1& = 0,
        \end{aligned}
    \end{equation}
    which can be regard as a \emph{generalized Ptolemy equation}.

\item The formula \eqref{f} satisfy the pentagon relation.
\end{enumerate}		
\end{cor}

\section{Weighted Delaunay inequality}\label{sec:del}

\subsection{Orthogonal circle formula}

Given an inversive distance function on edges $\mathbf{I} \in \mathbb{R}_{>1}^{E(\mathcal{T})}$ with respect to the triangulation $\mathcal{T}$ and a weight function $\mathbf{r} \in \mathbb{R}_{>0}^V$ as well on a marked surface $(S,V)$, we can construct a piecewise flat metric $d_f \in \Phi_\mathcal{T} \circ L_f(\mathbf{I},\mathbf{r})$. 

Firstly, we consider the inversive distance on a triangle face $f_{ijk}$. We denote that $a=I_{jk},b=I_{ki},c=I_{ij}$ of the face $f_{ijk}$ embedded in $\mathbb{E}^2$, and there exist a unique circle $\odot O_{ijk}$ orthogonal to the circles with centers of $v_\alpha$ and radii of $r_\alpha$ where $\alpha=i,j,k$. Denote the radius of $\odot O_{ijk}$ by $\rho_{ijk}$, then,

\begin{thm}\label{radius_area}
    \[
        r_ir_jr_k\sqrt{\Delta_{abc}}=2\rho_{ijk}\area(f_{ijk})
    \]
    holds when the face $f_{ijk}$ of inversive distance circle packing $L_f(\mathbf{I},\mathbf{r})$ satisfies the triangle inequality.
\end{thm}

\begin{proof}
    By establishing a Euclidean coordinate chart with the original point at the center of $\odot O_{ijk}$, denote the coordinate of vertices $v_\alpha=(x_\alpha,y_\alpha)$, then the inversive distances are
    
    \[
        I_{\alpha\beta}=\frac{(x_\alpha-x_\beta)^2+(y_\alpha-y_\beta)^2
        -r_\alpha^2-r_\beta^2}{2r_\alpha r_\beta}=			
        \frac{\rho_{ijk}^2-x_\alpha x_\beta-y_\alpha y_\beta}
        {\sqrt{x_\alpha^2+y_\alpha^2-\rho_{ijk}^2}
        \sqrt{x_\beta^2+y_\beta^2-\rho_{ijk}^2}}
    \]		
    where $\alpha \ne \beta \in \{i,j,k\}$. We can calculate that
    \[
        \Delta_{abc}=\frac{(x_iy_j-x_jy_i+x_jy_k-x_ky_j+x_ky_i-x_iy_k)^2
        \rho_{ijk}^2}{(x_i^2+y_i^2-1)(x_j^2+y_j^2-1)(x_k^2+y_k^2-1)}
        =\frac{4\area^2(f_{ijk})\rho_{ijk}^2}{r_i^2r_j^2r_k^2}
    \]
    and finish the proof by taking the square root at both side.
\end{proof}

\subsection{Triangle inequality}

Secondly, we consider the inversive distance on a hinge $\Diamond_{ij;kl}$ embedded in $\mathbb{E}^2$. In this subsection, in case of too many subscripts, we denote $P=v_k,Q=v_i,R=v_l,S=v_j$ for vertices, $p=r_k,q=r_i,r=r_l,s=r_j$ for radii, $\odot P,\odot Q,\odot R,\odot S$ for circles, and $a=I_{ki},b=I_{il},c=I_{lj},d=I_{jk},e=I_{ij}$ for inversive distances. If there is only one edge between two vertices without confusion, like $e_{ij}$ between $P$ and $Q$, then denote its length by $PQ$, etc.

\begin{figure}[ht]
    \centering
    \includegraphics{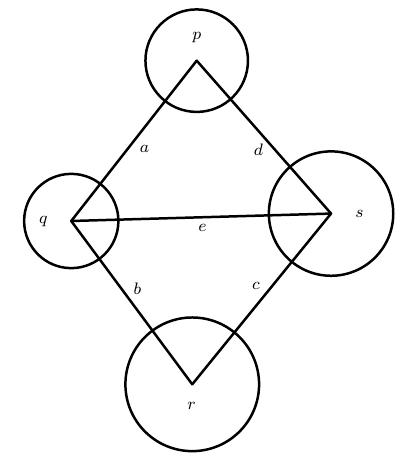}
    \caption{Notation of radii and inversive distances on a hinge}
    \label{fig:cir4}
\end{figure}

\begin{defn}\label{dft}

    Let $f$ satisfy the formula \eqref{f}, the \emph{local weighted Delaunay inequality} for the edge $e_{ij}$ is defined (in Euclidean case) by
    \begin{equation}\label{ineq_f}
        \frac{\sqrt {\Delta _{bce}}}{p} + \frac{\sqrt {\Delta _{ade}}}{r}\le
        \frac{\sqrt {\Delta _{cdf}}}{q} + \frac{\sqrt {\Delta _{abf}}}{s}
    \end{equation}

    The set of $(\mathbf{I},\mathbf{r})$ satisfied the inequality for edge $e_{ij}$ of the triangulation $\mathcal{T}$, denote by $D_f(e_{ij};\mathcal{T})$, is a closed set in $\mathbb{R}_{>1}^{E(\mathcal{T})} \times \mathbb{R}_{>0}^V$. 
    
    The intersection of all these $\left|E\right|$ number of closed sets with respect to $\mathcal{T}$ is denoted by
        \begin{equation}\label{ineq_f2}
        D_f(\mathcal{T}) \coloneqq \bigcap_{e_{ij} \in E(\mathcal{T})}
        D_f(e_{ij};\mathcal{T}) \subset \mathbb{R}_{>1}^{E(\mathcal{T})}
        \times \mathbb{R}_{>0}^V.
    \end{equation}

\end{defn}

\begin{lem}\label{trig_eq_lem}
	With the notation of symbols in this subsection on the hinge $\Diamond_{ij;kl}$, when inversive distances $a,b,c,d,e>1$ and radii $p,q,r,s>0$ satisfy the local weighted Delaunay inequality \eqref{ineq_f}, if the edge lengths $PQ,QR,RS,SP,QS$ calculated from $L_f$ satisfy $QR<QS+SR$ and $SR<QS+QR$, then $PQ+PS>QS$ holds.
\end{lem}

\begin{proof}
    Suppose that $PQ+PS>QS$ is false, then $PQ+PS \le QS$. Denote that
    \[
        \begin{aligned}
            G(p)\coloneqq&\frac12(PQ+PS-QS)(PQ+PS+QS)\\
            =&p^2+(aq+ds)p-eqs+\sqrt{p^2+2apq+q^2}\sqrt{p^2+2dps+s^2},
        \end{aligned}
    \]
	then $G(0+)=(1-e)qs<0$, $G(+\infty)=+\infty$, and $G$ strictly increases on $\mathbb{R}_{>0}$, so it has a unique zero point $G(p_0)=0$. 
	
	We define that
    \[
        p_0=\frac{(e^2-1)qs}{QS\sqrt{\Delta_{ade}}+(ae+d)q+(de+a)s},
    \]
	where $QS=\sqrt{q^2+s^2+2eqs}$. After simplifying and factoring, we find that
    \[
        (p_0^2+(aq+ds)p_0-eqs)^2=
        (p_0^2+2ap_0q+q^2)(p_0^2+2dp_0s+s^2)
    \]
    holds. Moreover,
    \[	
        \begin{aligned}
            &p_0^2+(aq+ds)p_0-eqs\\
            =&-\frac{qs\left(G'+(aq+2deq+ae^2q+ds+2aes+de^2s)
            QS\sqrt{\Delta_{ade}}\right)}{(QS\sqrt{\Delta_{ade}}
            +(ae+d)q+(de+a)s)^2}<0
        \end{aligned}
    \]
    where
    \[
        \begin{aligned}
            G'=&(ad+2a^2e+2d^2e+3ade^2+e^3-e)q^2\\
            +&(a^2+d^2+4ade+3a^2e^2+3d^2e^2+4ade^3+e^4-1)qs\\
            +&(ad+2a^2e+2d^2e+3ade^2+e^3-e)s^2>0.
        \end{aligned}
    \]
    Therefore,
    \[
        -(p_0^2+(aq+ds)p_0-eqs)=
        \sqrt{(p_0^2+2ap_0q+q^2)(p_0^2+2dp_0s+s^2)},
    \]		
    which is $G(p_0)=0$. Knowing that $G$ is monotonically increasing, we can conclude by $PQ+PS \le QS$ that $0<p \le p_0$.
    
    Next we discuss case by case. 
    For the first case, if $(b^2-1)q^2-2(bc+e)qs+(c^2-1)s^2 \le 0$, then
    \[
        \begin{aligned}
            &(beq+cq+ces+bs)^2-QS^2\Delta_{bce}\\
            =&(e^2-1)((b^2-1)q^2-2(bc+e)qs+(c^2-1)s^2)\le 0,
        \end{aligned}
    \]		
    which is $QS\sqrt{\Delta_{bce}} \ge beq+cq+ces+bs$. Considering the formula \eqref{delta} and $p \le p_0, r>0$, we know that
    \[
        \begin{aligned}
            &\frac{\sqrt{\Delta_{bce}}}{p}+\frac{\sqrt{\Delta_{ade}}}{r}-
            \left(\frac{\sqrt{\Delta_{cdf}}}{q}+\frac{\sqrt {\Delta _{abf}}}{s}\right)\\
            >& \frac{\sqrt{\Delta _{bce}}}{p_0}-
            \left(\frac{\sqrt{\Delta_{cdf}}}{q}+\frac{\sqrt{\Delta_{abf}}}{s}\right)\\
            =&\frac{\sqrt{\Delta_{ade}}}{(e^2-1)qs}(QS\sqrt{\Delta_{bce}}-(beq+cq+ces+bs)) \ge 0.
        \end{aligned}
    \]
    This contradicts the inequality \eqref{ineq_f}.

    For the second case, if $(b^2-1)q^2-2(bc+e)qs+(c^2-1)s^2>0$, then $QS\sqrt{\Delta_{bce}}<beq+cq+ces+bs$. We define that
    \[
         r_0 \coloneqq \frac{(e^2-1)qs}{(be+c)q+(ce+b)s-QS\sqrt{\Delta_{bce}}}
    \]
    which is slightly different from $p_0$.
    
    Similar to the discussion for $G(p)$, when $RQ+RS \le QS$ holds, we have
    \[
        0<r \le \frac{(e^2-1)qs}{QS\sqrt{\Delta_{bce}}+(be+c)q+(ce+b)s} < r_0.
    \]
    
    However, when $RQ+RS>QS$ holds, by simplifying the inequality
    \begin{equation}
        (QR+RS+QS)(QR+RS-QS)(RS+QS-QR)(QR+QS-RS)>0,
    \end{equation}
    we have
    \[
        \begin{aligned}
            -&((b^2-1)q^2-2(bc+e)qs+(c^2-1)s^2)r^2\\
            +&2(beq+cq+ces+bs)qsr-(e^2-1)q^2s^2>0.
        \end{aligned}
    \]
    This is a quadratic inequality of $r$, and the solution is
    \[
        \frac{(e^2-1)qs}{QS\sqrt{\Delta_{bce}}+(be+c)q+(ce+b)s}
        <r<r_0.
    \]
 
    Therefore, in the second case, no matter $RQ+RS \le QS$ or $RQ+RS>QS$, we have $0<r<r_0$ as long as $QR<QS+RS$ and $RS<QS+QR$ hold. Considering that $0<p \le p_0$, we have
    \[
        \frac{\sqrt {\Delta _{bce}}}{p} + \frac{\sqrt {\Delta _{ade}}}{r}>
        \frac{\sqrt {\Delta _{bce}}}{p_0} + \frac{\sqrt {\Delta _{ade}}}{r_0}=
        \frac{\sqrt {\Delta _{cdf}}}{q} + \frac{\sqrt {\Delta _{abf}}}{s}
    \]
    which also contradicts the inequality \eqref{ineq_f}. Thus, $PQ+PS>QS$ holds.
\end{proof}

\begin{thm}\label{del_in_trig_f}
    Given a triangulation $\mathcal{T}$ of a marked surface $(S,V)$, if there is an inversive distance circle packing $(\mathbf{I},\mathbf{r})$ such that the local weighted Delaunay inequality holds on every edge of $\mathcal{T}$, then all the faces satisfy the triangle inequality, and the polyhedral metric is well-defined by the circle packing. Namely,    
    \[
        D_f(\mathcal{T}) \subset Q_f(\mathcal{T}).
    \]
    
\end{thm}

\begin{proof}

    Suppose there exists $(\mathbf{I},\mathbf{r}) \in D_f(\mathcal{T})$ such that some faces do not satisfy the triangle inequality with the length of $\mathbf{l}=L_f(\mathbf{I},\mathbf{r})$. Then we pick the longest edge among the faces, denoted by $e_{ij}$. The face on one side of $e_{ij}$ not satisfied the triangle inequality is denoted by $f_{ijk}$, and the face on the other side is denoted by $f_{ijl}$. Then we have $l_{ij}+l_{jk}>l_{ik}$ and $l_{ij}+l_{ik}>l_{jk}$.

    If $f_{ijl}$ satisfies the triangle inequality, from lemma \ref{trig_eq_lem} we know that $l_{ij}<l_{jk}+l_{ik}$, then $f_{ijk}$ satisfies the triangle inequality, which is a contradiction.    
    
    If $f_{ijl}$ does not satisfy the triangle inequality, because of the longest edge length $l_{ij}$, we know that $l_{ij}+l_{jl}>l_{il}$ and $l_{ij}+l_{il}>l_{jl}$. Similarly, $l_{ij}<l_{jk}+l_{ik}$, then $f_{ijk}$ satisfies the triangle inequality, which is also a contradiction.

    Thus, for any $(\mathbf{I},\mathbf{r}) \in D_f(\mathcal{T})$ we know that $L_f(\mathbf{I},\mathbf{r})$ satisfies all the triangle inequality on every face of $\mathcal{T}$, and $L_f(D_f(\mathcal{T})) \subset \mathbb{R}^{E(\mathcal{T})}_\Delta$, which is $D_f(\mathcal{T}) \subset Q_f(\mathcal{T})$.
    
\end{proof}

\subsection{Weighted Delaunay condition for Euclidean case}

\begin{lem}\label{local_del_f_lem}
    With the notations above, Given a hinge $\Diamond_{ij;kl}$ satisfied the triangle inequalities, the edge $e_{ij}$ is local weighted Delaunay if and only if the formula \eqref{ineq_f} holds where $f$ satisfies \eqref{f}.
\end{lem}

\begin{proof}
    Substituting \eqref{delta} into \eqref{ineq_f}, we can simplify as follows
    \[
        \begin{aligned}
            &((be+c)qr+(ce+b)rs-(e^2-1)qs)p\sqrt{\Delta_{ade}}\\
            +&((ae+d)pq+(de+a)ps-(e^2-1)qs)r\sqrt{\Delta_{bce}} \ge 0.
        \end{aligned}
    \]
    Recall that $h_{ij,k}$ and $h_{ij,l}$ are the distance with sign between $e_{ij}$ and the centers of circles, and we can calculate that
    \[
        \begin{aligned}
            h_{ij,k}&=\frac{qs((ae+d)pq+(de+a)ps-(e^2-1)qs)}
            {2QS\,Area(f_{ijk})}\\
            h_{ij,l}&=\frac{qs((be+c)qr+(ce+b)rs-(e^2-1)qs)}
            {{2QS\,Area(f_{ijl})}}.
        \end{aligned}
    \]
    By Theorem \ref{radius_area}, we know that
    \[
        \begin{aligned}
            h_{ij,k}&=\frac{\rho_{ijk}((ae+d)pq+(de+a)ps-(e^2-1)qs)}
            {p\sqrt{q^2+2eqs+s^2}\sqrt{\Delta_{ade}}}\\
            h_{ij,l}&=\frac{\rho_{ijl}((be+c)qr+(ce+b)rs-(e^2-1)qs)}
            {r\sqrt{q^2+2eqs+s^2}\sqrt{\Delta_{bce}}}.
        \end{aligned}
    \]
    Therefore, the local weighted Delaunay inequality \eqref{ineq_f} is equivalent to
    \begin{equation}\label{h_over_rho}
        \frac{h_{ij,k}}{\rho_{ijk}}+\frac{h_{ij,l}}{\rho_{ijl}} \ge 0.
    \end{equation}

    From the simple Euclidean geometry relationship, the intersection points of the orthogonal circles $\odot O_{ijk}$ and $\odot O_{ijl}$ lies on $e_{ij}$. Denote the length of the intersection segment by $2h$, then $\rho_{ijk}=\sqrt{h^2+h_{ij,k}^2}$ and $\rho_{ijl}=\sqrt{h^2+h_{ij,l}^2}$. Since $\frac{x}{\sqrt{x^2 + h^2}}$ is a strictly increasing odd function, we know that the inequality \eqref{ineq_f} is equivalent to
    \[
        h_{ij,k} + h_{ij,l} \ge 0,
    \]
    which is local weighted Delaunay in Definition \ref{loc_del}.
\end{proof}

\begin{thm}\label{whole_del_f}
    The triangulation $\mathcal{T}$ of the piecewise flat surface $(S,V,d_f)$ is weighted Delaunay with respect to the weight $\mathbf{r} \in R$ if and only if there exist $(\mathbf{I},\mathbf{r}) \in D_f(\mathcal{T})$ such that $\Phi_\mathcal{T} \circ L_f(\mathbf{I},\mathbf{r})=[d_f]$.
\end{thm}

\begin{proof}
     On one hand, given $(S,V,d_f)$ and its weighted Delaunay triangulation $\mathcal{T}$, the lengths
    \[
        \mathbf{l}=\Phi_\mathcal{T}^{-1}([d_f]) = L_f(\mathbf{I},\mathbf{r}) \in \mathbb{R}^{E(\mathcal{T})}_\Delta
    \]
    satisfy the triangle inequalities. Then we solve inversive distances edge by edge from the above equality. As $\mathbf{r} \in R$, for any $v_i,v_j \in V$ and $e_{ij} \in E(\mathcal{T})$, we have $r_i,r_j>0$ and $r_i+r_j<l_{ij}$, thus $I_{ij}>1$ holds. In particular, If $e_{ii}$ is an edge connected the same vertex, then $l_{ii}>2\mathrm{Inj}(v_i)>2r_i$ and $I_{ii}>1$. By Theorem \ref{local_whole_f} and Lemma \ref{local_del_f_lem}, we know that the inversive distances and the radii satisfy every local weight Delaunay condition, thus $\tilde{L}_f^{-1}(\mathbf{l},\mathbf{r}) \in D_f(\mathcal{T})$.
    
    On the other hand, given $(\mathbf{l},\mathbf{r}) \in D_f(\mathcal{T})$, we know that $(\mathbf{I},\mathbf{r}) \in Q_f(\mathcal{T})$ by Theorem \ref{del_in_trig_f}, and $\tilde{L}_f(\mathbf{I},\mathbf{r})$ is well-defined. Also, by Lemma \ref{local_del_f_lem}, with the corresponding isometry class of piecewise flat metric $[d_f]=\Phi_\mathcal{T} \circ L_f(\mathbf{I},\mathbf{r})$, every hinge of the triangulation $\mathcal{T}$ is local weighted Delaunay. The last thing to verify is $\mathbf{r} \in R$. Given any vertex $v_0 \in V$, the faces around $v_0$ are denoted by $f_1,\dots,f_k$, and their orthogonal circles are denoted by $O_1,\dots,O_k$. For any vertex $v_i \ne v_0$, if there is an edge in $\mathcal{T}$ connecting $v_i$ and $v_0$, then $I_{0i}>0$ and $r_0+r_i<l_{0i}$. If not, for any $j \in \{1,\dots,k\}$, the circle centered at $v_i$ with radius $r_i$ does not intersect or intersect with angle no more than $\frac{\pi}{2}$. Note that $O_1,\dots,O_k$ intersect on edges around $v_0$, and the union of them deformation retract to $\partial Vor_f(v_0)$. Moreover, the intersection $\partial Vor_f(v_0) \cap O_j$ is two radii of $O_j$. Thus, the circle centered at $v_i$ with radius $r_i$ does not intersect with $\partial Vor_f(v_0)$, and contained in $S \setminus Vor_f(v_0)$, then $r_0+r_i<l_{0i}$ or $\mathbf{r} \in R$.
\end{proof}

\begin{rmk}
    If $a,b,c,d,e>1$ and let $f$ satisfies formula \eqref{f}, then the following four inequalities \eqref{usjm1}, \eqref{usjm2}, \eqref{usjm3} and \eqref{usjm4} are equivalent and hold with equality simultaneously.
    \begin{equation}\label{usjm1}
        (1+a+d-e)\sqrt{(1+b)(1+c)}+(1+b+c-e)\sqrt{(1+a)(1+d)} \ge 0
    \end{equation}
    \begin{equation}\label{usjm2}
        (1+a+d-e)\sqrt{\Delta_{bce}}+(1+b+c-e)\sqrt{\Delta_{ade}} \ge 0
    \end{equation}   
    \begin{equation}\label{usjm3}
        \sqrt{(1+a)(1+d)}+\sqrt{(1+b)(1+c)}\ge\sqrt{(1+e)(1+f)}
    \end{equation}
    \begin{equation}\label{usjm4}
        \sqrt{\Delta_{abf}}+\sqrt{\Delta_{cdf}}\ge\sqrt{\Delta_{ade}}+\sqrt{\Delta_{bce}}
    \end{equation}
\end{rmk}
    
With the notation above, Ushijima proved that $e_{ij}$ is a Delaunay edge (with equal weights) of a truncated triangulation of hyperbolic surface with geodesic boundaries if and only if inequality \eqref{usjm1} or \eqref{usjm3} holds in \cite{ushijima1999canonical}, and we prove that $e_{ij}$ is Delaunay if and only if inequality \eqref{usjm2} or \eqref{usjm4} holds, where \eqref{usjm2} can be easily obtained from \eqref{usjm4} and \eqref{delta}.

\subsection{Weighted Delaunay condition for Hyperbolic case}

\begin{lem}\label{diameter}
    See Figure \ref{fig:ratio}. If an isometric embedding from the hyperbolic plane $\mathbb{H}^2$ in Lemma \ref{gbh_sinh} to a Poincar\'e unit disk $\mathbb{D}$, such that the geodesic $\gamma$ maps to the diameter of $\mathbb{D}$, then the ratio of the radii of the Euclidean arc where $\gamma_1$ and $\gamma_2$ locate is $r_1 \colon r_2$. Moreover, the inversive distance between the Euclidean circles where $\gamma_1$ and $\gamma_2$ locate is $\cosh d_\mathbb{H}(\gamma_1,\gamma_2)$.
\end{lem}

\begin{figure}[ht]
    \centering
    \includegraphics{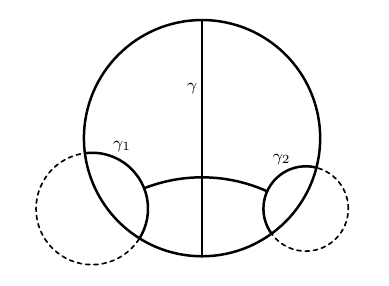}
    \caption{The geodesic $\gamma$ like a diameter satisfies Equation \eqref{vor_sinh}.} \label{fig:ratio}
\end{figure}

\begin{proof}
    In this proof, the Euclidean distance between two points is denoted as $d_\mathbb{E}(P,Q)=\overline{PQ}$.

    Denote the centers of Euclidean circles where $\gamma_1$ and $\gamma_2$ locate by $O_1$ and $O_2$, and their radii by $\rho_1$ and $\rho_2$. The line segment connecting $O_1$ and $O_2$ intersect $\partial\mathbb{D}$, $\gamma_1$, $\gamma$, $\gamma_2$ and $\partial\mathbb{D}$ at points $A_1$, $B_1$, $M$, $B_2$, and $A_2$ in order. By the projective invariance of cross ratio and the hyperbolic length formula of cross ratio, we have
    \[
        d_\mathbb{H}(\gamma,\gamma_1)=\log\frac
        {\overline{MA_1}\,\overline{B_1A_2}}
        {\overline{A_1B_1}\,\overline{A_2M}}\quad
        d_\mathbb{H}(\gamma,\gamma_2)=\log\frac
        {\overline{MA_2}\,\overline{B_2A_1}}
        {\overline{A_2B_2}\,\overline{A_1M}}.
    \]

    Note that the length of tangent line from the center of $\mathbb{D}$ to $\gamma_1$ and $\gamma_2$ is $1$, then $\gamma$ is the radical axis of these two circles. Let $l=\overline{O_1O_2}$, and solve the equations $(\rho_1+\overline{B_1M})^2-\rho_1^2=(\rho_2+\overline{B_2M})^2-\rho_2^2$ and $l=\rho_1+\overline{B_1M}+\overline{MB_2}+\rho_2$. The solution is 
    \[
        \overline{B_1M}=\frac{(l-\rho_1)^2-\rho_2^2}{2l}\quad
        \overline{B_2M}=\frac{(l-\rho_2)^2-\rho_1^2}{2l}.
    \]
    By $\overline{OM}^2=1-\overline{A_1M}^2=1+\rho_1^2-(\rho_1+\overline{B_1M})^2$, we have
    \[
        A_1M=A_2M=\frac{\sqrt{(l-\rho_1-\rho_2)(l+\rho_1-\rho_2)
        (l-\rho_1+\rho_2)(l+\rho_1+\rho_2)}}{2l}.
    \]
    Then we can calculate the hyperbolic sine of $d_\mathbb{H}(\gamma,\gamma_1)$ and $d_\mathbb{H}(\gamma,\gamma_2)$. That is
    \[
        \begin{aligned}
            \sinh d_\mathbb{H}(\gamma,\gamma_1)=&
            \frac{\sqrt{(l-\rho_1-\rho_2)(l+\rho_1-\rho_2)
            (l-\rho_1+\rho_2)(l+\rho_1+\rho_2)}}{2l\rho_1},\\
            \sinh d_\mathbb{H}(\gamma,\gamma_2)=&
            \frac{\sqrt{(l-\rho_1-\rho_2)(l+\rho_1-\rho_2)
            (l-\rho_1+\rho_2)(l+\rho_1+\rho_2)}}{2l\rho_2}.\\
        \end{aligned}
    \]
    Therefore, $r_1 : r_2=\sinh d_\mathbb{H}(\gamma,\gamma_2):\sinh d_\mathbb{H}(\gamma,\gamma_1)=\rho_1 : \rho_2$.

    Finally, we verify that
    \[
        \cosh d_\mathbb{H}(\gamma_1,\gamma_2)=
        \cosh (d_\mathbb{H}(\gamma_1,\gamma)+d_\mathbb{H}(\gamma,\gamma_2))=
        \frac{l^2-\rho_1^2-\rho_2^2}{2\rho_1\rho_2}
    \]
    is the inversive distance.
\end{proof}	

Given a marked surface $(S,V)$ and its related hyperbolic surface with geodesic boundaries $(\Sigma,d)$, with a weight function $\mathbf{r} \in \mathbb{R}^V_{>0}$ and a truncated triangulation $\mathcal{T}=(\Gamma,E,F)$. Denote that $\mathbf{l}=\Omega_\mathcal{T}^{-1}([d])$ and $l_{ij}=\mathbf{l}(e_{ij})$.

For a truncated hinge $\Diamond_{ij;kl}$, we denote that $p=r_k,q=r_i,r=r_l,s=r_j,a=\cosh l_{ki},b=\cosh l_{il},c=\cosh l_{lj},d=\cosh l_{jk},e=\cosh l_{ij}$.

\begin{lem}\label{gbh_del}
    With the notations above, the edge $e_{ij}$ is local weighted Delaunay if and only if the inequality \eqref{ineq_f} holds, where $f$ satisfies the formula \eqref{f}.
\end{lem}

\begin{proof}
    We consider the hyperbolic right-angled hexagon $f_{ijk} \subset \mathbb{D}$ conformally embedded in the unit circle in $\mathbb{E}^2$. Since the length $h_{ij,k}$ is hyperbolic, we denote $h'_{ij,k}$ as the Euclidean length of the face $f_{ijk}$ with radii $p,q,s$ and inversive distances $a,d,e$. Moreover, the radius of the orthogonal circle of $f_{ijk}$ is denoted by $\rho'_{ijk}$. By Lemma \ref{diameter},  the image of $f_{ijk}$ in $\mathbb{E}^2$ is an arc right-angled hexagon with radii $\frac{p}{\rho'_{ijk}},\frac{q}{\rho'_{ijk}},\frac{s}{\rho'_{ijk}}$ and inversive distances $a,d,e$. By simple calculations, we obtain
    \[
        \tanh h_{ij,k}=\frac{h'_{ijk}}{\rho'_{ijk}}.
    \]
    Therefore, the inequality \eqref{ineq_f} is equivalent to \eqref{h_over_rho} or $\tanh h_{ij,k}+\tanh h_{ij,l} \ge 0$ here. Since $\tanh$ is an increasing odd function, the inequality \eqref{ineq_f} is equivalent to $h_{ij,k}+h_{ij,l} \ge 0$, which means the edge $e_{ij}$ is local weighted Delaunay.
\end{proof}

The judgment of global weighted Delaunay on hyperbolic surface with geodesic boundaries is decided by whether the truncated triangulation is the dual of the weighted Voronoi decomposition or not. Since every edge is local weighted Delaunay, it is the dual of some 1-cell of weighted Voronoi decomposition. The duality from 2-cells and 0-cells to vertices and faces is guarantee by Theorem \ref{gbh_unique}. Similar to Theorem \ref{whole_del_f}, we have the following theorem.
\begin{thm}\label{gbh_whole_del}
    Given a hyperbolic surface with geodesic boundaries $(\Sigma,\Gamma,d)$ and a weight function $\mathbf{r} \in \mathbb{R}^\Gamma_{>0}$, the truncated triangulation $\mathcal{T}$ is weighted Delaunay if and only if there exist $(\mathbf{I},\mathbf{r}) \in D_f$ such that $\Omega_\mathcal{T} \circ \arcch(\mathbf{I})=[d]$, where
    \[
        \arcch \colon \mathbb{R}^k_{\ge 1} \to \mathbb{R}^k_{\ge 0}\quad
        (I_1,\dots,I_k)\mapsto(\arcch I_1,\dots,\arcch I_k).
    \]
\end{thm}

\section{Diffeomorphism between Teichm\"uller spaces}\label{sec:diffeo}

\subsection{Homeomorphism between Teichm\"uller spaces}

Recall that given a marked surface $(S,V)$, the Teichm\"uller space of piecewise flat metric is
\[
    Teich_f(S,V)=\bigcup_\mathcal{T}\Phi_\mathcal{T}
    (\mathbb{R}^{E(\mathcal{T})}_\Delta).
\]
Then we attach the weight $R$ to Teichm\"uller spaces. Here we do not omit the metric on the symbol, and denote it as $R(d_f)$. Define that
\[
    \widetilde{Teich}_f(S,V) \coloneqq \left\{([d_f],\mathbf{r})
    \mid [d_f]\in Teich_f(S,V), \mathbf{r} \in R(d_f) \right\},
\]
and the map
\[
    \tilde{\Phi}_\mathcal{T} \colon Q_f(\mathcal{T}) \subset
    \mathbb{R}_{>1}^{E(\mathcal{T})} \times \mathbb{R}_{>0}^V \to
    Teich_f(S,V) \times \mathbb{R}_{>0}^V \quad (\mathbf{I},\mathbf{r}) 
    \mapsto (\Phi_\mathcal{T} \circ L_f(\mathbf{I},\mathbf{r}),\mathbf{r}).
\]

This map is similar to $\tilde{L}_f$, however, there image lie in different space. The tilde symbol here does not mean that $\tilde{\Phi}_\mathcal{T}$ is the lifting of  $\Phi_\mathcal{T}$ for concise. Actually, it is the composition of the lifting with $L_f$.

Similar to $P_f(\mathcal{T})=\Phi_\mathcal{T}(\mathbb{R}^{E(\mathcal{T})}_\Delta)$, we define
\begin{equation}\label{P_f}
    \tilde P_f(\mathcal{T}) \coloneqq \tilde\Phi_\mathcal{T}(Q_f(\mathcal{T})).
\end{equation}

Since $\Phi_\mathcal{T} \colon \mathbb{R}^{E(\mathcal{T})}_\Delta \to P_f(\mathcal{T})$ is a homeomorphism, by Proposition \ref{lift_inject}, we know $\tilde\Phi_\mathcal{T} \colon Q_f(\mathcal{T}) \to \tilde P_f(\mathcal{T})$ is a real analytic homeomorphism.

Let $\pi([d_f],\mathbf{r})=[d_f]$. Since $L_f(Q_\mathcal{T})=\mathbb{R}^{E(\mathcal{T})}_\Delta$, we know $\pi(\tilde P_f(\mathcal{T}))=\Phi_\mathcal{T} \circ L_f(Q_f(\mathcal{T}))=P_f(\mathcal{T})$.

Recall the Teichm\"uller space of hyperbolic surface with geodesic boundaries as well. Suppose a surface $\Sigma$ with $\chi(\Sigma)<0$ and a marked surface $(S,V)$ are related, given a triangulation $\mathcal{T}=(V,E,F)$, for any length coordinate $\mathbf{x} \in \mathbb{R}_{>0}^{E(\mathcal{T})}$, there is an equivalent class up to an isometry isotopic to the identity on $\Sigma$ with a representative element $d$, which is $[d]=\Omega_\mathcal{T}(\mathbf{x}) \in Teich(\Sigma)$.

To attach the weight on the Teichm\"uller space, we define
\begin{equation}\label{omega_}
    \tilde{\Omega}_\mathcal{T} \colon \mathbb{R}_{>1}^{E(\mathcal{T})}
    \times \mathbb{R}_{>0}^V \to Teich(\Sigma) \times 
    \mathbb{R}_{>0}^V \quad (\mathbf{I},\mathbf{r}) \mapsto
    (\Omega_\mathcal{T}\circ \arcch(\mathbf{I}),\mathbf{r}),
\end{equation}
which is a homeomorphic map by Theorem \ref{teich home formula}.

With these notations, we know $\tilde\Phi_\mathcal{T}(D_f(\mathcal{T}))\subset\widetilde{Teich}_f(S,V)$ by Theorem \ref{del_in_trig_f} and Theorem \ref{whole_del_f}. Considering Theorem \ref{local_whole_f} and Theorem \ref{gbh_unique}, we have
\begin{equation}\label{union_teich_f}
    \begin{aligned}
        \widetilde{Teich}_f(S,V)&=\bigcup_\mathcal{T}
        \tilde\Phi_\mathcal{T}(D_f(\mathcal{T}))
        \\
        Teich(\Sigma) \times \mathbb{R}_{>0}^V&=\bigcup_\mathcal{T}
        \tilde\Omega_\mathcal{T}(D_f(\mathcal{T})).
    \end{aligned}
\end{equation}

Now we define
\begin{equation}\label{A_T}%
    A_\mathcal{T} \coloneqq \tilde\Omega_\mathcal{T} \circ 
    \tilde\Phi_\mathcal{T}^{-1} \colon \tilde P_f(\mathcal{T})
    \to Teich(\Sigma) \times \mathbb{R}_{>0}^V,
\end{equation}
which is a composition of two real analytic homeomorphism. Since the image of $\tilde\Phi_\mathcal{T}^{-1}$ is a subset of the domain of $\tilde\Omega_\mathcal{T}$, we know $A_\mathcal{T}$ is a real analytic injective map.

For two different triangulations $\mathcal{T}$ and $\mathcal{T}'$, if $P_f(\mathcal{T})\cup P_f(\mathcal{T}') \ne \varnothing$, let $([d_f],\mathbf{r})\in\tilde P_f(\mathcal{T})\cup P_f(\mathcal{T}')$, generally, $A_\mathcal{T}([d_f],\mathbf{r}) \ne A_{\mathcal{T}'}([d_f],\mathbf{r})$. However, we have the following lemma:
\begin{lem}\label{intersection}
    Given two triangulations $\mathcal{T}$ and $\mathcal{T}'$ of the surface $(S,V)$ such that $\tilde\Phi_\mathcal{T}(D_f(\mathcal{T})) \cap \tilde\Phi_{\mathcal{T}'}(D_f(\mathcal{T}')) \ne \varnothing$, suppose $([d_f],\mathbf{r}) \in \tilde\Phi_\mathcal{T}(D_f(\mathcal{T})) \cap \tilde\Phi_{\mathcal{T}'}(D_f(\mathcal{T}'))$, then
    \[
        A_\mathcal{T}([d_f],\mathbf{r})=A_{\mathcal{T}'}([d_f],\mathbf{r}).
    \]
\end{lem}

\begin{proof}
    By Theorem \ref{whole_del_f} we know that $\mathcal{T}$ and $\mathcal{T}'$ are all weight Delaunay triangulation of $(S,V,d_f)$ with the weight $\mathbf{r}$. By Theorem \ref{local_whole_f} we know that they differ by a finite number of edge switches.

    Without loss of generality, we suppose that $\mathcal{T}$ and $\mathcal{T}'$ differ by one switch of the hinge $\Diamond_{ij;kl}$. With the notations in Formula \eqref{ineq_f}, denote the inversive distance components of $\Phi_\mathcal{T}^{-1}([d],\mathbf{r})$ at edges $e_{ki},e_{il},e_{lj},e_{jk},e_{ij}$ by $a,b,c,d,e$, and the weight components at vertices $v_k,v_i,v_l,v_j$ by $p,q,r,s$, then the components of $\Phi_{\mathcal{T}'}^{-1}([d],\mathbf{r})$ at $e_{ki},e_{il},e_{lj},e_{jk},v_k,v_i,v_l,v_j$ are equal to the ones of $\Phi_\mathcal{T}^{-1}([d],\mathbf{r})$, while the one at $e_{kl}$ is denoted by $f$. By Theorem \ref{teich home formula}, to prove $A_\mathcal{T}([d_f],\mathbf{r})=A_{\mathcal{T}'}([d_f],\mathbf{r})$, we only need to show that $f$ satisfies Formula \eqref{f}.

    The hinge $\Diamond_{ij;kl}$ being able to switch means $O_{ijk}$ and $O_{ijl}$ coincide. After developing them onto $\mathbb{E}^2$, the four circles at vertices share a common orthogonal circle. We zoom the orthogonal circle to the unit disk and give a Poincar\'e hyperbolic metric on it. The arcs truncated by the orthogonal circle, denoted by $\Gamma_k,\Gamma_i,\Gamma_l,\Gamma_j$ in counter-clockwise order, zoomed at the same time, can be regarded as geodesics in $\mathbb{D}$. The radical axes of these arcs pass through $O_{ijk}$ or $O_{ijl}$. By Lemma \ref{diameter}, the hyperbolic distances between these arcs are $\arcch a,\arcch b,\arcch c,\arcch d,\arcch e$. (See Figure \ref{fig:4hyp}.) By Theorem \ref{teich home formula}, the hyperbolic distance between $\Gamma_k$ and $\Gamma_l$ is $\arcch f$, where $f$ satisfies Equation \eqref{f}. Also, by Lemma \ref{diameter}, the inversive distance between $\Gamma_k$ and $\Gamma_l$ is $f$. Since inversive distance does not change by zooming, in the flipped hinge $\Diamond_{kl;ij}$ from $\Diamond_{ij;kl}$, the inversive distance of edge $e_{kl}$ is $f$ as desired.

    For the case that $\mathcal{T}$ and $\mathcal{T}'$ differ by a finite number of switches, just repeat the step above. The image of $([d_f],\mathbf{r})$ does not change under different mappings.
\end{proof}

We can glue all the mappings $A_\mathcal{T}$ together, and construct the mapping between two different kinds of Teichm\"uller space of $(S,V)$. Define that
\[
    \mathbf{A} \coloneqq \bigcup_\mathcal{T}A_\mathcal{T}|_
    {\tilde\Phi_\mathcal{T}(D_f(\mathcal{T}))} \colon \widetilde{Teich}_f(S,V)
    \to Teich(\Sigma) \times \mathbb{R}_{>0}^V
\]
By Lemma \ref{intersection} and Equation \eqref{union_teich_f}, $\mathbf{A}$ is well-defined.

All the relation about the mappings are concluded in Figure \ref{fig:maps}.

\begin{figure}
    \centering
    \begin{tikzcd}[row sep=tiny,column sep=small]
        Teich(\Sigma)\times\mathbb{R}_{>0}^V \arrow[dddddddd] &  &  &  &  &
            {\widetilde{Teich}_f(S,V)} \arrow[dddddddd, bend left]
            \arrow[lllll, "\mathbf{A}"'] \\
        &  &  &  &  & \cup \\
        &  & Q_f(\mathcal{T}) \arrow[rrdd, "\tilde L_f"]
            \arrow[rrr, "\tilde\Phi_\mathcal{T}"]  &  &  & 
            \tilde P_f(\mathcal{T}) \arrow[llllluu, "A_\mathcal{T}"'] \\
        &  & \cap  &  &  &  \\
        &  & \mathbb{R}_{>1}^{E(\mathcal{T})}\times\mathbb{R}_{>0}^V
            \arrow[rrdddd, "L_f"] \arrow[dddd, "\pi"] 
            \arrow[lluuuu, "\tilde{\Omega}_\mathcal{T}"'] &  & 
            \mathbb{R}^{E(\mathcal{T})}_\Delta\times\mathbb{R}_{>0}^V
            \arrow[dd, "\pi"] & \\
        &  &  &  &  &  \\[60pt]
        &  &  &  & \mathbb{R}^{E(\mathcal{T})}_\Delta
            \arrow[r, "\Phi_\mathcal{T}"]  & P_f(\mathcal{T}) \\
        &  &  &  & \cap & \cap  \\
        Teich(\Sigma) & \mathbb{R}_{>0}^{E(\mathcal{T})}
            \arrow[l, "\Omega_\mathcal{T}"'] & \mathbb{R}_{>1}^{E(\mathcal{T})}
            \arrow[l, "\arcch"'] &  & \mathbb{R}_{>0}^{E(\mathcal{T})} & {Teich_f(S,V)}
    \end{tikzcd}
    \caption{Notations of mappings}\label{fig:maps}
\end{figure}
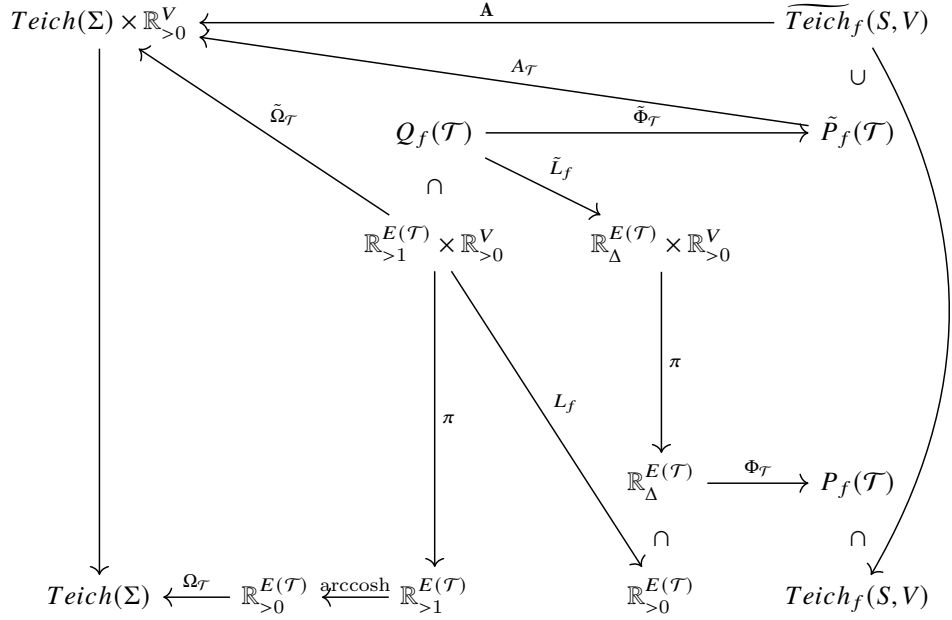

\begin{lem}\label{A_home}
    $\mathbf{A}$ is homeomorphic.
\end{lem}

\begin{proof}
    Since $\mathbf{A}$ is a piecewise real analytic map such that the image at the intersections of the pieces are the same, we know that $\mathbf{A}$ is continuous. To prove that $\mathbf{A}$ is homeomorphic, we directly construct the inverse map of $\mathbf{A}$ and prove that the inverse map is continuous.

    Define
    \[
        A'_\mathcal{T} \coloneqq \tilde\Phi_\mathcal{T} \circ 
        \tilde\Omega_\mathcal{T}^{-1} \colon \tilde\Omega_\mathcal{T}(D_f(\mathcal{T}))
        \to Teich_f(S,V) \times \mathbb{R}_{>0}^V.
    \]
    By Proposition \ref{lift_inject}, we know that this map is real analytic injective.
    From Equation \eqref{union_teich_f}, we have $\image A'_\mathcal{T}=\tilde\Phi_\mathcal{T}(D_f(\mathcal{T})) \subset \widetilde{Teich}_f(S,V)$.
   
    Following Lemma \ref{intersection}, let $\mathcal{T}$ and $\mathcal{T}'$ be two truncated triangulations of $\Sigma$ such that $\tilde\Omega_\mathcal{T}(D_f(\mathcal{T})) \cap \tilde\Omega_{\mathcal{T}'}(D_f(\mathcal{T}')) \ne \varnothing$. 

    Let $([d],\mathbf{r})$ belong to the intersection mentioned above, then both $\mathcal{T}$ and $\mathcal{T}'$ are truncated weighted Delaunay triangulations of the hyperbolic surface $(\Sigma,d)$ with weights $\mathbf{r}$.
    According to Theorem \ref{gbh_whole_del}, $\mathcal{T}$ and $\mathcal{T}'$ differ by a finite number of edge switches.
    Without loss of generality, assume that they differ only by one switch on the hinge $\Diamond_{ij;kl}$. Then in the Voronoi decomposition, the dual 0-cell $O_{ijk}=O_{ijl}\in\Sigma$ of the two faces in the hinge coincides.

    Embed the universal cover of $\Sigma$ into the Poincaré disk $\mathbb{D}$ such that $O_{ijk}$ maps to the Euclidean center of $\mathbb{D}$.
    Then the hinge $\Diamond_{ij;kl}$ is a hyperbolic right-angled octagon in $\mathbb{D}$. 
    According to Lemma \ref{diameter}, the ratio of the radii of the Euclidean circles on which the boundaries $\Gamma_k,\Gamma_i,\Gamma_l,\Gamma_j$ lie is $r_k:r_i:r_l:r_j$, and the inversive distance between the circles is the hyperbolic cosine of the hyperbolic length. 
    Connect the corresponding centers of these four Euclidean circles to form the quadrilateral hinge denoted by $\Diamond'_{ij;kl}$.
    It can be verified that the piecewise flat surface of $A'_\mathcal{T}([d],\mathbf{r})$ restricted on $\Diamond_{ij;kl}$, when developed and embedded into $\mathbb{E}^2$, forms a quadrilateral similar to $\Diamond'_{ij;kl}$ up to a scaling.

    When we switch the hinge $\Diamond_{ij;kl}$ to obtain $\Diamond_{kl;ij}$, and embed the universal cover of $\Sigma$ into $\mathbb{D}$ such that $O_{kli}$ maps to the center of the unit circle. 
    This gives us another quadrilateral $\Diamond'_{kl;ij}$ formed by connecting the centers of the Euclidean circles where the boundaries lie. By the geometry relation, $\Diamond_{ij;kl}$ and $\Diamond'_{kl;ij}$ are congruent up to a rotation.

    Similarly, the piecewise flat surface of $A'_{\mathcal{T}'}([d],\mathbf{r})$ restricted on $\Diamond_{kl;ij}$ developed and embedded into $\mathbb{E}^2$ is similar to $\Diamond'_{kl;ij}$.
    Thus, the hinges $\Diamond'_{ij;kl}$ and $\Diamond'_{kl;ij}$ are similar.
    Furthermore, since both $A'_\mathcal{T}([d],\mathbf{r})$ and $A'_{\mathcal{T}'}([d],\mathbf{r})$ share the same radii, the corresponding radii of $\Diamond'_{ij;kl}$ and $\Diamond'_{kl;ij}$ are the same. Therefore, $\Diamond'_{ij;kl}$ and $\Diamond'_{kl;ij}$ are congruent, and then $A'_\mathcal{T}([d],\mathbf{r})=A'_{\mathcal{T}'}([d],\mathbf{r})$.

    Define $\mathbf{A}' \coloneqq \bigcup_\mathcal{T}A'_\mathcal{T}$, then $\mathbf{A}'$ is piecewise real analytic and continuous at the interfaces of the pieces. By Equation \eqref{union_teich_f}, we know that $\mathbf{A}' \colon Teich(\Sigma) \times \mathbb{R}_{>0}^V \to \widetilde{Teich}_f(S,V)$ is globally continuous. By definition, we can verify that $\mathbf{A}' \circ \mathbf{A} = id_{\widetilde{Teich}_f(S,V)}$, so both $\mathbf{A}$ and $\mathbf{A}'$ are homeomorphic. Namely, the spaces $\widetilde{Teich}_f(S,V)$ and $Teich(\Sigma) \times \mathbb{R}_{>0}^V$ are homeomorphic.
\end{proof}

\subsection{Cell decomposition}

In this subsection, we introduce the result by Schlenker and Dimitrov in paper \cite{schlenker2008circle} and \cite{dimitrov2015hyper}. They proved that $D_f(\mathcal{T})$ in Definition \ref{dft} is simply connected, moreover,
\begin{prop}\label{cell}
    Equation \eqref{union_teich_f} forms a local finite cell decomposition.
\end{prop}

The inequalities to define the domain like $D_f(\mathcal{T})$ as \eqref{ineq_f} are highly nonlinear. It is hard to know if $D_f(\mathcal{T})$ is simply connected. However, if we find a real analytic coordinate transformation from $(\mathbf{I},\mathbf{r})$ to some coordinate, then $D_f(\mathcal{T})$ is homeomorphic to a polytope defined by a linear programming, which is simply connected. Moreover, the equations of the linear inequalities define the boundary operator.

Schlenker and Dimitrov chose the coordinate $(\mathbf{K},\boldsymbol{\theta})$, which is 
\[
    \mathbf{K} \colon V \to (-\infty,2\pi) \quad v_i \mapsto K_i
    \qquad \boldsymbol{\theta} \colon E \to (0,\pi) \quad
    e_{ij} \mapsto \theta_{ij}
\]
where $K_i$ is the discrete curvature of vertex $v_i$, and $\theta_{ij}$ is the intersection angle of orthogonal circles $\odot O_{ijk}$ and $\odot O_{ijl}$ on edge $e_{ij}$. An inversive distance circle packing with weighted Delaunay triangulation $\mathcal{T}$, or a \emph{hyper-ideal circle pattern} in their papers, is well-defined and unique up to scaling.

The coordinate $(\mathbf{K},\boldsymbol{\theta})$ should satisfy some linear inequalities, which are local Gauss-Bonnet formula on some domains. We omit the detail here. Note that the cell decomposition is uniformly local finite, because for any inversive distance circle packing, the ways to switch the triangulation in the case of sharing a common orthogonal circle is finite.

\subsection{First derivative of the diffeomorphism}

\begin{lem}\label{dfdF}
    Let the hinge $\Diamond_{ij;kl}$ be assigned with radii $p,q,r,s>0$ at vertices $v_k,v_i,v_l,v_j$, and with inversive distances $a,b,c,d,e>1$ at edges $e_{ki}$, $e_{il}$, $e_{lj}$, $e_{jk}$, $e_{ij}$. 
    The corresponding edge lengths of the inversive distance circle metric computed by $L_f$ are denoted as $u,v,w,x,y$. 
    If the faces $f_{ijk}, f_{ijl}$ satisfy the triangle inequalities, we develop and embed $\Diamond_{ij;kl}$ into $\mathbb{E}^2$, and let $z = d_\mathbb{E}(v_k,v_l)$.
    
    Let $F=\frac{z^2-p^2-r^2}{2pr}$, and let $f$ satisfy Equation \eqref{f}. Consider $F$ and $f$ as functions of variables $p,q,r,s,a,b,c,d,e$. When the local weighted Delaunay inequality \eqref{ineq_f} holds with equality, we have
    \[
        dF=df.
    \]
\end{lem}

\begin{proof}
    Recall the notations in Theorem \ref{teich home formula}, in the developed hinge $\Diamond_{ij;kl}$, denoted the inner angle of the face $f_{ijk}$ at the vertex $v_i$ by $\alpha$, and the one of the face $f_{ijl}$ by $\beta$. By cosine law,
    \[
        \cos \alpha = \frac{u^2+y^2-x^2}{2uy}\quad
        \cos \beta = \frac{v^2+y^2-w^2}{2vy} \quad
        \cos(\alpha+\beta)=\frac{u^2+v^2-z^2}{2uv}.
    \]
    Substitute them in the identity $(\cos(\alpha+\beta)-\cos\alpha\cos\beta)^2=(1-\cos^2\alpha)(1-\cos^2\beta)$ and simplify, we have
    \[
        \begin{aligned}
            0=&u^4 w^2+u^2 w^4+v^4 x^2+v^2 x^4+y^4 z^2+y^2 z^4\\
            +&u^2 x^2 y^2+v^2 w^2 y^2+u^2 v^2 z^2+w^2 x^2 z^2\\
            -&u^2 v^2 w^2-u^2 v^2 x^2-u^2 w^2 x^2-v^2 w^2 x^2-w^2 y^2 z^2-x^2 y^2 z^2\\
            -&v^2 x^2 y^2-u^2 w^2 y^2-u^2 w^2 z^2-v^2 x^2 z^2-u^2 y^2 z^2-v^2 y^2 z^2.
        \end{aligned}
    \]

    Substituting the formulas for the edge length, i.e., $u^2=p^2+q^2+2apq$, and the special one $z^2=p^2+r^2+2Fpr$, into the above equation, we have
    \[
        X(p,q,r,s,a,b,c,d,e,F)=0,
    \]
    where $X$ is a ten-variable polynomial of degree eight with each variable having a maximum degree of two. The specific form is too lengthy to be included here.

    Denote the generalized Ptolemy equation \eqref{abcdef} by $Y(a,b,c,d,e,f)=0$. Let
    \begin{equation}
        \begin{aligned}
            p_0 \coloneqq&\frac{\sqrt\Delta _{bce}}{\dfrac{\sqrt\Delta_{cdf}}{q}
            +\dfrac{\sqrt\Delta_{abf}}{s}-\dfrac{\sqrt\Delta_{ade}}{r}}\\
            =&\frac{qrs\sqrt\Delta _{bce}}
            {\frac{(d+ae)\sqrt{\Delta_{bce}}+(c+be) \sqrt{\Delta_{ade}}}{e^2-1}qr+
            \frac{(a+de)\sqrt{\Delta_{bce}}+(b+ce)\sqrt{\Delta_{ade}}}{e^2-1}rs-
            \sqrt{\Delta_{ade}}qs}.
        \end{aligned}
    \end{equation}

    When the inequality \eqref{ineq_f} holds in equation and $f$ satisfies Equation \eqref{f}, we have $p=p_0(q,r,s,a,b,c,d,e)$. What is more important, by Lemma \ref{intersection}, the circles centered at vertices $v_k,v_i,v_l,v_j$ with radii $p,q,r,s$ and inversive distances $a,b,c,d,e,f$ share a common orthogonal circle. This picture means that even through $F \ne f$ in usual, we have $F=f(a,b,c,d,e)$ in this situation.

    Note that $X$ is a quadratic function of $p$, so there exists a unique critical point $p_1=p_1(q,r,s,a,b,c,d,e,F)$ such that $X_p(p_1,q,r,s,a,b,c,e,d,F)=0$. Substituting $F=f(a,b,c,d,e)$ into $p_1$, we have $p_1=p_1(q,r,s,a,b,c,d,e)$. Note that $p_0=p_0(q,r,s,a,b,c,d,e)$, we can verify that $p_0=p_1$ using a symbolic computation software. See codes in appendix.
    By the formula for implicit derivative $F_p=-X_p/X_F$, after substituting $F=f$ and $p=p_1$, we have $F_p=0=f_p$.

    Substituting $F=f$ and $p=p_1$ into $X_q$, $Y_a/Y_f-X_a/X_F$ and $Y_e/Y_f-X_e/X_F$, respectively, and combining them by finding a common denominator and factoring, we see that all the numerators contain the factor $Y(a,b,c,d,e,f)$. Since $f$ satisfies $Y(a,b,c,d,e,f)=0$, the three results are equal to $0$. Then we have $F_q=-X_q/X_F=0$, $Y_a/Y_f-X_a/X_F=0$ and $Y_e/Y_f-X_e/X_F=0$, that is to say $F_q=f_q=0$,$F_a=f_a$ and $F_e=f_e$.

    Using the symmetry of the variables in the expressions for $X$ and $Y$, we can derive other equalities between partial derivatives: $F_r=f_r=F_s=f_s=0$, and $F_b=f_b,F_c=f_c,F_d=f_d$. Thus, Lemma $dF=df$ is proved.
\end{proof}

\begin{thm}\label{c1diff}
    $\mathbf{A}$ is a $C^1$ diffeomorphism.
\end{thm}

\begin{proof}
    For any $([d_f],\mathbf{r}) \in \widetilde{Teich}_f(S,V)$, if there exists a triangulation $\mathcal{T}$ such that $([d_f],\mathbf{r}) \in \interior \tilde\Omega_\mathcal{T}(D_f(\mathcal{T}))$, we know that $\mathbf{A}$ is smooth near $\tilde\Omega_\mathcal{T}([d_f],\mathbf{r})$ because $A_\mathcal{T}$ is real analytic.

    If not, there still exists a triangulation $\mathcal{T}$ such that $([d_f],\mathbf{r}) \in \tilde\Omega_\mathcal{T}(D_f(\mathcal{T}))$. In this case, since $\tilde\Omega_\mathcal{T}$ is homeomorphic, we have $([d_f],\mathbf{r}) \in \partial\tilde\Omega_\mathcal{T}(D_f(\mathcal{T})) = \tilde\Omega_\mathcal{T}(\partial D_f(\mathcal{T}))$. 
    By the real analytic cell decomposition structure in Proposition \ref{cell}, there exist a finite number of triangulations denoted by $\mathcal{T}_1=\mathcal{T},\dots,\mathcal{T}_k,k \ge 2$, such that they differ from $\mathcal{T}$ by a finite number of edge switches, which means $([d_f],\mathbf{r}) \in \tilde\Omega_\mathcal{T}(\bigcap_{i=1}^k D_f(\mathcal{T}_i))$.
    Moreover, there exist a neighborhood $U$ such that $([d_f],\mathbf{r}) \in U \subset \tilde\Omega_\mathcal{T}(\bigcup_{i=1}^k D_f(\mathcal{T}_i))$. To prove that $\mathbf{A}$ is $C^1$ near $\tilde\Omega_\mathcal{T}([d_f],\mathbf{r})$, it is sufficient to show that $dA_{\mathcal{T}}=dA_{\mathcal{T}'}$ at $\tilde\Omega_\mathcal{T}([d_f],\mathbf{r})$ for any $\mathcal{T}' \in \left\{\mathcal{T}_2,\dots,\mathcal{T}_k\right\}$.
    
    Consider $\tilde\Omega_{\mathcal{T}'} \circ \tilde\Omega_\mathcal{T}^{-1}$ and $\tilde\Phi_{\mathcal{T}'} \circ \tilde\Phi_\mathcal{T}^{-1}$.
    Firstly, we show that $d(\tilde\Omega_{\mathcal{T}'} \circ \tilde\Omega_\mathcal{T}^{-1}) = d(\tilde\Phi_{\mathcal{T}'} \circ \tilde\Phi_\mathcal{T}^{-1})$ holds at $([d_f],\mathbf{r})$, if $\mathcal{T}$ and $\mathcal{T}'$ are differed by one edge switch.
    By observing the coordinates of these two maps, we see that they are identity on every coordinate except the switched one. The two different inversive distance values (i.e., $f$ and $F$ in Lemma \ref{dfdF}) have the same first-order partial derivatives with respect to other coordinates by Lemma \ref{dfdF}.
    Thus, the two maps $\tilde\Omega_{\mathcal{T}'} \circ \tilde\Omega_\mathcal{T}^{-1}$ and $\tilde\Phi_{\mathcal{T}'} \circ \tilde\Phi_\mathcal{T}^{-1}$ have the same Jacobi matrices.
    Then, by $\tilde\Omega_{\mathcal{T}'}$ and $\tilde\Phi_\mathcal{T}$ are real analytic homeomorphic, we know that $dA_{\mathcal{T}}=dA_{\mathcal{T}'}$ holds at $\tilde\Omega_\mathcal{T}([d_f],\mathbf{r})$.
    For $\mathcal{T}$ and $\mathcal{T}'$ differed by finite switches, just repeat the steps above for finite times.

    Therefore, we have proved that $\mathbf{A}$ is a global $C^1$ diffeomorphism. It can be verified that $\mathbf{A}$ is not a $C^2$ diffeomorphism, but this is irrelevant and omitted.
\end{proof}

\section{Proof of the main theorem}\label{sec:proof}

\subsection{Discrete conformal equivalence}
The discrete conformal equivalent class of inversive distance circle packing on polyhedral surface is defined as follows.

\begin{defn}
    Given a marked surface $(S,V)$, suppose $d_f$ and $d_f'$ are two piecewise flat metrics with a legal weight function $\mathbf{r} \in R(d_f)$ and $\mathbf{r}' \in R(d_f')$ respectively, we say $([d_f],\mathbf{r})$ and $([d_f'],\mathbf{r}')$ are \emph{discrete conformal equivalent} for inversive distance circle packing, if the first component of $\mathbf{A}([d_f],\mathbf{r})$ and $\mathbf{A}([d_f'],\mathbf{r}')$ are equal, which means two derived hyperbolic metrics with geodesic boundaries are isotopic.
\end{defn}

Since $\mathbf{A}$ is homeomorphic, this definition of equivalence has reflexivity, symmetry, and transitivity, thus it is well-defined. Then we have the following proposition.

\begin{prop}
    Given a marked surface $(S,V)$ and its related surface $\Sigma$, for any discrete conformal equivalent class, it can be represented by
    \[
        \mathbf{A}^{-1}\left(\left\{[d]\right\} \times \mathbb{R}_{>0}^V\right)
        \subset \widetilde{Teich}_f(S,V),
    \]
    where $[d]$ is an isometry class isotopic to identity on the hyperbolic surface with geodesic boundaries. 
\end{prop}

Here is the main theorem of this paper.

\begin{thm}\label{main_f}
    Given a marked surface $(S,V)$ with $n$ vertices, for any $([d_f],\mathbf{r}) \in \widetilde{Teich}_f(S,V)$ and target discrete curvature
    \[
        \bar{\mathbf K} \colon V \to (-\infty,2\pi) \quad v_i \mapsto \bar K_i
    \]
    satisfies the Gauss-Bonnet formula $\sum_{i=1}^n\bar K_i=2\pi \chi(S)$, there exist $([d_f'],\mathbf{r}') \in \widetilde{Teich}_f(S,V)$ discrete conformal equivalent to $([d_f],\mathbf{r})$, such that the discrete curvature of the piecewise flat metric $d_f'$ at the $v_i \in V$ is equal to $\bar K_i$. This tuple of metric and radii is unique up to global scaling.
\end{thm}

\subsection{Variational principle}

The following Lemma and Theorem can be found in \cite{guo2011local} and \cite{zhang2014unified}.

\begin{lem}\label{diff_trig}
    Given a marked surface $(S,V)$ with a triangulation $\mathcal{T}=(V,E,F)$, for any $(\mathbf{I},\mathbf{r}) \in Q_f(\mathcal{T})$, construct the inversive distance circle packing $L_f(\mathbf{I},\mathbf{r})=\mathbf{l}$. Denote the inner angle of $f_{ijk}$ at vertices $v_i,v_j,v_k$ by $\theta_i,\theta_j,\theta_k$, and the opposite edge lengths by $l_i,l_j,l_k$. Let $u_i= \log r_i$ for every $v_i$, following the notation of $h_{ij,k}$ in definition \ref{loc_del}, the Jacobi matrix
    \[
        \frac{\partial(\theta_i,\theta_j,\theta_k)}{\partial(u_i,u_j,u_k)}=
        \left(
            \begin{matrix}
    -\frac{h_{j,ik}}{l_j}-\frac{h_{k,ij}}{l_k} & \frac{h_{k,ij}}{l_k} & \frac{h_{j,ik}}{l_j} \\
    \frac{h_{k,ij}}{l_k} & -\frac{h_{i,jk}}{l_i}-\frac{h_{k,ij}}{l_k} & \frac{h_{i,jk}}{l_i} \\
    \frac{h_{j,ik}}{l_j} & \frac{h_{i,jk}}{l_i} & -\frac{h_{i,jk}}{l_i}-\frac{h_{j,ik}}{l_j}
            \end{matrix}
        \right)
    \]
    is symmetric, semi-negative definite, with a one-dimensional null space and eigenvector $(1,1,1)$.
\end{lem}

Since $K_i=2\pi-\sum_{v_i \in f} \theta_i$, by the lemma above, we have the following theorem.

\begin{thm}\cite{guo2011local}\cite{zhang2014unified}\label{diff_all}
   
    Given a piecewise flat surface $(S,V,d_f)$ and weight $\mathbf{r} \in R_2(d_f)$, denote the discrete curvature at vertices by $\mathbf{K}$, then the weighted Delaunay triangulation $\mathcal{T}$ satisfies $([d_f],\mathbf{r}) \in \tilde\Phi_\mathcal{T}(D_f(\mathcal{T}))$.
    Denote its dual weighted Voronoi decomposition by $\mathcal{C}=(V^*,E^*,F^*)$. Let $l_{ij}=\mathbf{l}(e_{ij})$, then the length of the dual edge $e_{ij}^* \in E^*$ is denoted by $l^*_{ij}$. The elements of Jacobi matrix $\frac{\partial \mathbf{K}}{\partial \mathbf{u}}$ is
    \[
        \left(\frac{\partial \mathbf{K}}{\partial \mathbf{u}}\right)_{ij}=
        \begin{cases}
            -\dfrac{l^*_{ij}}{l_{ij}} & e_{ij} \in E\\
            \sum_{e_{ik} \in E}\dfrac{l^*_{ik}}{l_{ik}} & j=i\\
            0 & \text{otherwise},
        \end{cases}         
    \]
    where $\mathbf{u}=\log(\mathbf{r})$. This matrix is symmetric, semi-negative definite, with a one-dimensional null space and eigenvector $(1,\dots,1) \in \mathbb{R}^V$.

\end{thm}

The Jacobi matrix above is a sparse matrix formed by combining matrices from Lemma \ref{diff_trig}. If there is a face glued by itself in the triangulation, we can simply add the elements at the corresponding positions of the matrix, and the conclusion remains unchanged.

Lemma \ref{diff_trig} and Theorem \ref{diff_all} describe the local differential properties. With the help of Theorem \ref{c1diff}, we can define the curvature map globally.

\begin{defn}\label{kappa}
    Suppose $\Sigma$ is the related surface of $(S,V)$ with $n$ vertices. Given a hyperbolic metric with geodesic boundaries $d$ on $\Sigma$, define that  
    \[
        \kappa_d \colon \mathbb{R}^n \to (-\infty,2\pi)^n \quad
        \mathbf{u} \mapsto \mathbf{K}=(K_1,\dots,K_n),
    \]
    where $\mathbf{K}$ is the discrete curvature of the metric at the first component in $\mathbf{A}^{-1}([d],\exp(\mathbf{u}))$. 
    
    Moreover, define the \emph{Ricci potential} as
    \[
        \mathcal{E}_d=\mathcal{E}_d(\mathbf{w}) \coloneqq \int ^\mathbf{w} \sum_{i=1}^n K_i\,du_i.
    \]
\end{defn}

Considering that $\mathbf{A}$ is $C^1$ and the map from edge length to curvature is real analytic by cosine law, we know $\kappa_d$ is $C^1$ as a restriction of the composition of these two maps. Note that the domain of $\mathbf{K}$ here is expanded than the one in Theorem \ref{diff_all}. To show that the Ricci potential is well-defined, let
\[
    U_i\coloneqq \left\{\, \mathbf{u}\in\mathbb{R}^V \mid
    \tilde\Omega_{\mathcal{T}_i} \left( [d],\exp(\mathbf{u}) \right)
        \subset D_f(\mathcal{T}_i) \,\right\}.
\]
Zhu proved that the number of the set like this finite in \cite{zhu2019discrete}. Denote them by $U_1,\dots,U_M$, then
\begin{equation}\label{rn_cell}
    \mathbb{R}^V=\bigcup_{j=1}^M U_j.
\end{equation}

By Formula \eqref{ineq_f} and \eqref{ineq_f2}, every $U_i$ is real analytic homeomorphic to an $n$-dimensional convex polytope, and some of them intersect to get a low dimension one. Then we have
\begin{prop}
    Formula \ref{rn_cell} form a finite CW decomposition of $\mathbb{R}^V$.
\end{prop}

From Theorem \ref{diff_all}, we have $\frac{\partial K_i}{\partial u_j}=\frac{\partial K_j}{\partial u_i}$ holds on each $U_i$, and since $K_i$ is $C^1$ continuous, it follows that $\frac{\partial K_i}{\partial u_j}=\frac{\partial K_j}{\partial u_i}$ holds on the entire simply connected $\mathbb{R}^V$. Therefore, the differential form $\sum_{i=1}^n K_i du_i$ is a closed form and hence an exact form, and its integral is independent of the choice of path. Thus, $\mathcal{E}_d$ is well-defined on $\mathbb{R}^V$. The reason for not including an integration starting point is because the Ricci energy could differ by a constant.

\begin{prop}\label{convex}
    $\mathcal{E}_d$ is a $C^2$ function on $\mathbb{R}^V$, and it is a strictly convex function restricted on the hyperplane
    \[U \coloneqq \left\{\, \mathbf{u} \in \mathbb{R}^V 
    \mid u_1+\dots+u_n=0 \,\right\}.\]
\end{prop}

\begin{proof}
    We know that $\nabla \mathcal{E}_d(\mathbf{u})=\kappa_d(\mathbf{u})$ is $C^1$ by definition, then $\mathcal{E}_d$ is $C^2$. The matrix $\hess \mathcal{E}_d$ is continuous on $\mathbb{R}^V$.

    From Theorem \ref{diff_all}, we know that $\hess \mathcal{E}_d$ is symmetric positive semi-definite on each $U_1,\dots,U_m$ and has only one-dimensional null space with the eigenvector $(1,\dots,1)$. Moreover, since it is globally continuous on $\mathbb{R}^V$, it is also symmetric positive semi-definite on the entire space $\mathbb{R}^V$ and has only one-dimensional null space with the eigenvector $(1,\dots,1)$. Therefore, it is strictly positive definite on the quotient space $U$, and hence $\mathcal{E}_d$ is strictly convex.
\end{proof}

Now we prove the main theorem of this paper.

\begin{proof}[Proof of Theorem \ref{main_f}]  
    By Proposition \ref{convex}, we know $\kappa_d|_U$ is injective from variational principle. Denote that
    \[
        K=\left\{\, \mathbf{K}=(K_1,\dots,K_n) \in (-\infty,2\pi)^V \mid \sum_{i=1}^n K_i=2\pi \chi(S) \,\right\}.
    \]

    Apparently, $K$ a bounded open set on the $(n-1)$-dimensional hyperplane, and the points on its boundary must have at least one component equal to $2\pi$.
    By the construction of $\kappa_d$ we know it is a continuous map and $\image \kappa_d \subset K$. By Brouwer's invariance of domain theorem, the map $\kappa_d|_U \colon U \to K$ is a continuous injective map between real $n-1$ dimensional topological manifolds, thus it is an open map. 

    We aim to prove that for any infinite sequence $\left\{\mathbf{u}^{(m)}\right\} \subset U$ satisfying
    \begin{equation}\label{infty}
        \lim_{k \to \infty} \mathbf{u}^{(m)}=\mathbf{u}^\infty
        \in [-\infty,+\infty]^n \setminus U,
    \end{equation}
    there exists a subsequence $\left\{\mathbf{u}^{(m_i)}\right\}$ such that
    \[
        \lim_{i \to \infty}\kappa_d(\mathbf{u}^{(m_i)}) \in \partial K
    \]

    From equation \eqref{rn_cell}, we observe that $U=\bigcup_{j=1}^M U_j\cap U$ is a finite cell decomposition. By the pigeonhole principle, there exists some cell $U_j \subset \tilde \Omega_\mathcal{T}(D_f(\mathcal{T}))$ that contains infinitely many elements from $\left\{\mathbf{u}^{(m)}\right\}$. Without loss of generality, we selected them as a subsequence, but still use the notation $\left\{\mathbf{u}^{(m)}\right\}$ for concise.
    Then we discuss it on the triangulation $\mathcal{T}=\mathcal{T}_j$.

    Consider $\mathbf{u}^\infty=(u^\infty_1,\dots,u^\infty_n)$, where there are several components that approach $-\infty$. Let $V_{good} \subsetneq V$ be the non-empty set of vertices corresponding to these components, which are denoted as \emph{good vertices}. The remaining vertices that are not good are denoted as \emph{bad vertices}. Note that the limit of the components of bad vertices may be a finite number or $+\infty$ as well. It is guaranteed by equation \eqref{infty} that $V_{good}$ is not empty and not $V$, since points on the infinite point of the hyperplane $U$ must have components diverge to $+\infty$ or $-\infty$.

    Firstly, we observe that a triangle $f_{ijk}$ in $\mathcal{T}$ cannot contain exactly one edge with two vertices being bad (note that $v_i$, $v_j$, and $v_k$ may coincide). 
    Otherwise, without loss of generality, assume that $v_i$ is a good while $v_j$ and $v_k$ are bad. During the convergence of $\{\mathbf{u}^{(m)}\}$, since $r_j^{(m)},r_k^{(m)} \ge \epsilon>0$, when $m$ is sufficiently large, we have $l_{ij}^{(m)}=\sqrt{(r_i^{(m)})^2+(r_j^{(m)})^2+2I_{ij}^{(m)}r_i^{(m)}r_j^{(m)}} \to r_j^\infty$ and $l_{ik}^{(m)} \to r_k^\infty$. Moreover, $l_{jk}^{(m)} \to r_j^\infty+r_k^\infty+\delta$ with $\delta>0$. Thus, when $m$ is sufficiently large, the face $f_{ijk}$ breaks the triangle inequality $l_{ij}^{(m)}+l_{ik}^{(m)}>l_{jk}^{(m)}$, which contradicts $\left\{\mathbf{u}^{(m)}\right\} \subset \tilde \Omega_\mathcal{T}(D_f(\mathcal{T})) \subset \tilde \Omega_\mathcal{T}(Q_f(\mathcal{T}))$.
    
    Secondly, Since $S$ is connected, we can select an edge that connects a good vertex to a bad vertex. All the neighborhood of the bad vertex must be good, otherwise, we could find a neighboring triangle that contradicts the previous argument.

    Finally, by the formula of the inverse distance circle packing, the curvature at this bad vertex converges to $2\pi$ as $k \to \infty$. Namely, $\kappa_d(\mathbf{u}^{(m)})$ converges to the boundary of $K$ up to the subspace topology.
    
    Therefore, we have proved that $\image \kappa_d=K$.
    Otherwise, suppose $\mathbf{K}_0 \in K \setminus\image \kappa_d$, since $\image \kappa_d$ is not empty, we can connect a path $\gamma \colon I \to K$ from the interior of $\image \kappa_d$ to $\mathbf{K}_0$. Because $\image \kappa_d$ is open and simply connected, there exist $s \in I$ such that for any sufficiently large $m$, we have $\gamma(s-1/m) \in \image \kappa_d$ and $\gamma([s,1])\in K\setminus\image\kappa_d$. Let $\mathbf{u}^{(m)}=\kappa_d^{-1}(\gamma(s-1/m))$, then $\mathbf{u}^{(m)}$ converge to the boundary of $U$, however, of which any subsequence can not converge to the boundary of $K$, which is a contradiction. Thus, we get that $\kappa_d$ is a bijection.

    The condition of the theorem requires that $\bar{\mathbf K} \in K$, so there exist a unique preimage $u=\kappa_d|_U^{-1}(\bar{\mathbf K})$. Therefore, for any $\forall t \in \mathbb{R}$, the piecewise flat surface with radii as weights $([d_f],\mathbf{r})$ is discrete conformal equivalent to $\mathbf{A}^{-1}([d],\mathbf{u}+(t,\dots,t))$, where $d$ satisfies $([d],\mathbf{r})=\mathbf{A}([d_f],\mathbf{r})$. The coordinate $\mathbf{u}+(t,\dots,t)$ the piecewise flat surface is up to a scaling. This is the end of the proof.

\end{proof}

\subsection{Discrete Ricci flow}

The \emph{discrete Ricci flow} of $2$-dimensional polyhedral surface with inversive distance circle packing is defined by
\begin{equation}\label{ricciflow}
    \frac{dr_i}{dt}=-(K_i-\bar K)r_i
\end{equation}
By Theorem \ref{main_f}, this flow can be extended.

\begin{thm}\label{ricci_flow}
    For any initial $\mathbf{r}(0)=\mathbf{r}_0 \in \mathbb{R}_{>0}$, the ODE \eqref{ricciflow} have the solution $\mathbf{r}(t)$ existing on $\left[0,+\infty\right)$, whose limitation $\mathbf{r}^\infty \coloneqq \lim\limits_{t \to +\infty}\mathbf{r}(t)$ satisfies that on every vertex $v_i$, the discrete curvature of $A^{-1}([d],\mathbf{r}^\infty)$ is $\bar K$
\end{thm}

\begin{proof}
    Let $\mathbf{u}(t)\coloneqq \log\left(\mathbf{r}(t)\right)$, then ODE \eqref{ricciflow} is
    \[
        \frac{d\mathbf{u}(t)}{dt}=-(\mathbf{K}-\bar{\mathbf{K}})
    \]

    Define the \emph{normalized} Ricci potential as
    \[
        \mathcal{E}_d(\mathbf{w})=\int ^\mathbf{w} \sum_{i=1}^n (K_i-K_{avg})\,du_i.
    \]
    Take the derivative with respect to $t$.
    \[
        \begin{aligned}
            \frac{d\mathcal{E}_d(\mathbf{u}(t))}{dt}= & 
            \frac{\partial\mathcal{E}_d}{\partial u_1}\frac{du_1}{dt}+\dots+
            \frac{\partial\mathcal{E}_d}{\partial u_n}\frac{du_n}{dt} \\
            = & -\left(\frac{\partial\mathcal{E}_d}{\partial u_1}(K_1-\bar K_1)+\dots+
            \frac{\partial\mathcal{E}_d}{\partial u_n}(K_n-\bar K_n)\right)\\
            = & -((K_1-\bar K_1)^2+\dots+(K_n-\bar K_n)^2) \le 0
        \end{aligned}
    \]

    Then $\mathcal{E}_d(\mathbf{u}(t))$ decrease and strictly decrease when $\mathbf{u}(t)) \in U$ and $\mathbf{K}(t) \ne \bar{\mathbf{K}}$.

    By Theorem \ref{main_f}, there exist $\mathbf{u}^\infty \in U$ such that $\kappa_d(\mathbf{u}^\infty)=\mathbf{K}_{avg}$, and $\nabla\mathcal{E}_d(\mathbf{u}^\infty)=0$. Thus, $\mathbf{u}^\infty$ is the unique critical point in $U$, and the unique minimum point by convexity.
    Therefore, $\mathcal{E}_d(\mathbf{u}(t))$ is bounded from below, and the solution can be extended infinitely. When $t \to \infty$, we have $\frac{d\mathcal{E}_d(\mathbf{u}(t))}{dt} \to 0$ and $\mathbf{K}(+\infty) = \bar{\mathbf{K}}$.
\end{proof}

\appendix

\section{Hyperbolic triangle law}

See figure \ref{fig:hyp6}. Denote the length of a hyperbolic right angle hexagon by $x,c,y,a,z,b$ in order, and the hyperbolic cosine law \cite{mondello2009triangulated} is
\begin{equation}\label{cosh}
    \cosh a= \frac{\cosh y \cosh z + \cosh x}{\sinh y \sinh z}.
\end{equation}

\begin{figure}[ht]
    \centering
    \includegraphics{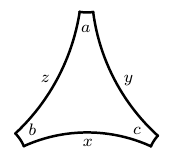}
    \caption{A hyperbolic right angle hexagon}
    \label{fig:hyp6}
\end{figure}

\section{Codes}

The codes for lemma \ref{intersection} and lemma \ref{dfdF} is below. The software is Mathematica 8.0.4.0.
\begin{verbatim}
X=Expand[
u^4w^2+u^2w^4+v^4x^2+v^2x^4+y^4z^2+y^2z^4
+u^2x^2y^2+v^2w^2y^2+u^2v^2z^2+w^2x^2z^2
-u^2v^2w^2-u^2v^2x^2-u^2w^2x^2-v^2w^2x^2
-w^2y^2z^2-x^2y^2z^2-v^2x^2y^2-u^2w^2y^2
-u^2w^2z^2-v^2x^2z^2-u^2y^2z^2-v^2y^2z^2
/.{u->Sqrt[p^2+2a p q+q^2],v->Sqrt[q^2+2b q r+r^2],
w->Sqrt[r^2+2c r s+s^2],x->Sqrt[s^2+2d s p+p^2],
y->Sqrt[q^2+2e q s+s^2],z->Sqrt[p^2+2F p r+r^2]}];
Y=a^2+b^2+c^2+d^2+e^2+f^2-a^2c^2-b^2d^2-e^2f^2-1+
2(a d e+b c e+a b f+c d f+a b c d+a c e f+b d e f);
p0=q r s Factor[Sqrt[b^2+c^2+e^2+2b c e-1]/(
((d+a e)Sqrt[b^2+c^2+e^2+2b c e-1]+
(c+b e)Sqrt[a^2+d^2+e^2+2a d e-1])/(e^2-1)q r
-Sqrt[a^2+d^2+e^2+2a d e-1]q s
+((a+d e)Sqrt[b^2+c^2+e^2+2b c e-1]+
(b+c e)Sqrt[a^2+d^2+e^2+2a d e-1])/(e^2-1)r s)];
p1=(p/.Solve[D[X,p]==0,p][[1]]);
temp=Factor[p0/(p1/.
F->(a b+c d+a c e+b d e+Sqrt[b^2+c^2+e^2+2b c e-1]
Sqrt[a^2+d^2+e^2+2a d e-1])/(-1 + e^2))];
Expand[Numerator[temp]]-Expand[Denominator[temp]]
Factor[D[X,q]/.p->p1/.F->f]
Factor[D[Y,a]/D[Y,f]-D[X,a]/D[X,F]/.p->p1/.F->f]
Factor[D[Y,e]/D[Y,f]-D[X,e]/D[X,F]/.p->p1/.F->f]
\end{verbatim}

The output is $0$ and three fractions with factor $Y(a,b,c,d,e,f)$ in their numerators.

\bibliographystyle{alpha}
\bibliography{ref}

\newcommand{\etalchar}[1]{$^{#1}$}
\begin{thebibliography}{SWGL15}

\bibitem[BH03]{bowers2003planar}
Philip~L Bowers and Monica~K Hurdal.
\newblock Planar conformal mappings of piecewise flat surfaces.
\newblock In {\em Visualization and mathematics III}, pages 3--34. Springer, 2003.

\bibitem[BL23]{bobenko2023decorated}
Alexander~I. Bobenko and Carl O.~R. Lutz.
\newblock Decorated discrete conformal maps and convex polyhedral cusps, 2023.

\bibitem[BPS15]{bobenko2015discrete}
Alexander~I Bobenko, Ulrich Pinkall, and Boris~A Springborn.
\newblock Discrete conformal maps and ideal hyperbolic polyhedra.
\newblock {\em Geometry \& Topology}, 19(4):2155--2215, 2015.

\bibitem[BS04a]{bobenko2004variational}
Alexander Bobenko and Boris Springborn.
\newblock Variational principles for circle patterns and {K}oebe's theorem.
\newblock {\em Transactions of the American Mathematical Society}, 356(2):659--689, 2004.

\bibitem[BS04b]{bowers2004uniformizing}
Philip~L Bowers and Kenneth Stephenson.
\newblock Uniformizing dessins and {B}elyimaps via circle packing.
\newblock {\em Memoirs of the American Mathematical Society}, 170(805), 2004.

\bibitem[BS07]{bobenko2007discrete}
Alexander~I Bobenko and Boris~A Springborn.
\newblock A discrete {L}aplace--{B}eltrami operator for simplicial surfaces.
\newblock {\em Discrete \& Computational Geometry}, 38(4):740--756, 2007.

\bibitem[CL{\etalchar{+}}03]{chow2003combinatorial}
Bennett Chow, Feng Luo, et~al.
\newblock Combinatorial {R}icci flows on surfaces.
\newblock {\em Journal of Differential Geometry}, 63(1):97--129, 2003.

\bibitem[CLXZ22]{chen2022bowersstephensons}
Yuxiang Chen, Yanwen Luo, Xu~Xu, and Siqi Zhang.
\newblock Bowers-stephenson's conjecture on the convergence of inversive distance circle packings to the riemann mapping, 2022.

\bibitem[DGL08]{dai2008variational}
Junfei Dai, Xianfeng~David Gu, and Feng Luo.
\newblock {\em Variational principles for discrete surfaces}, volume~4.
\newblock International Press of Boston Incorporated, 2008.

\bibitem[Dim15]{dimitrov2015hyper}
Nikolay Dimitrov.
\newblock Hyper-ideal circle patterns with cone singularities.
\newblock {\em Results in Mathematics}, 68(3-4):455--499, 2015.

\bibitem[GGL{\etalchar{+}}18]{gu2018discreteII}
Xianfeng Gu, Ren Guo, Feng Luo, Jian Sun, and Tianqi Wu.
\newblock A discrete uniformization theorem for polyhedral surfaces ii.
\newblock {\em Journal of Differential Geometry}, 109(3):431--466, 07 2018.

\bibitem[GJ17a]{ge2017deformationII}
Huabin Ge and Wenshuai Jiang.
\newblock On the deformation of inversive distance circle packings, ii.
\newblock {\em Journal of Functional Analysis}, 272(9):3573--3595, 2017.

\bibitem[GJ17b]{ge2017deformation}
Huabin Ge and Wenshuai Jiang.
\newblock On the deformation of inversive distance circle packings, iii.
\newblock {\em Journal of Functional Analysis}, 272(9):3596--3609, 2017.

\bibitem[GJ19]{ge2019deformation}
Huabin Ge and Wenshuai Jiang.
\newblock On the deformation of inversive distance circle packings, i.
\newblock {\em Transactions of the American Mathematical Society}, 2019.

\bibitem[Gli08]{glickenstein2008geometric}
David Glickenstein.
\newblock Geometric triangulations and discrete {L}aplacians on manifolds.
\newblock {\em arXiv: Metric Geometry}, 2008.

\bibitem[GLSW18]{gu2018discrete}
Xianfeng Gu, Feng Luo, Jian Sun, and Tianqi Wu.
\newblock A discrete uniformization theorem for polyhedral surfaces.
\newblock {\em Journal of Differential Geometry}, 109(2):223--256, 06 2018.

\bibitem[GLW19]{gu2019convergence}
David Gu, Feng Luo, and Tianqi Wu.
\newblock Convergence of discrete conformal geometry and computation of uniformization maps.
\newblock {\em Asian Journal of Mathematics}, 23(1):21--34, 2019.

\bibitem[Gor11]{gorlina2011weighted}
Yuliya Gorlina.
\newblock {\em Weighted {D}elaunay Triangulations of Piecewise-Flat Surfaces}.
\newblock PhD thesis, The University of Arizona, 2011.

\bibitem[Guo11]{guo2011local}
Ren Guo.
\newblock Local rigidity of inversive distance circle packing.
\newblock {\em Transactions of the American Mathematical Society}, 363(9):4757--4776, 2011.

\bibitem[Hat91]{hatcher1991triangulations}
Allen Hatcher.
\newblock On triangulations of surfaces.
\newblock {\em Topology Appl}, 40(2):189--194, 1991.

\bibitem[Hat02]{hatcher2002algebraic}
A.~Hatcher.
\newblock {\em Algebraic Topology}.
\newblock Cambridge University Press, 2002.

\bibitem[HS96]{he1996convergence}
Zheng-Xu He and Oded Schramm.
\newblock On the convergence of circle packings to the {R}iemann map.
\newblock {\em Inventiones mathematicae}, 125(2):285--305, 1996.

\bibitem[Koe36]{koebe1936kontaktprobleme}
Paul Koebe.
\newblock {\em Kontaktprobleme der konformen Abbildung}.
\newblock Hirzel, 1936.

\bibitem[L{\etalchar{+}}07]{luo2007teichmuller}
Feng Luo et~al.
\newblock On {T}eichm{\"u}ller spaces of surfaces with boundary.
\newblock {\em Duke Mathematical Journal}, 139(3):463--482, 2007.

\bibitem[Luo11]{luo2011rigidity}
Feng Luo.
\newblock Rigidity of polyhedral surfaces, iii.
\newblock {\em Geometry \& Topology}, 15(4):2299--2319, 2011.

\bibitem[Lut22]{lutz2022canonical}
Carl O.~R. Lutz.
\newblock Canonical tessellations of decorated hyperbolic surfaces.
\newblock {\em Geometriae Dedicata}, 217(2):14, Dec 2022.

\bibitem[LWZ21]{luo2021convergence}
Yanwen Luo, Tianqi Wu, and Xiaoping Zhu.
\newblock The convergence of discrete uniformizations for genus zero surfaces, 2021.

\bibitem[MMP87]{mitchell1987the}
Joseph S~B Mitchell, David~M Mount, and Christos~H Papadimitriou.
\newblock The discrete geodesic problem.
\newblock {\em SIAM Journal on Computing}, 16(4):647--668, 1987.

\bibitem[Mon09]{mondello2009triangulated}
Gabriele Mondello.
\newblock Triangulated {R}iemann surfaces with boundary and the {W}eil-{P}etersson {P}oisson structure.
\newblock {\em Journal of Differential Geometry}, 81(2):391--436, 2009.

\bibitem[Mos88]{Mosher1988Tiling}
Lee Mosher.
\newblock Tiling the projective foliation space of a punctured surface.
\newblock {\em Transactions of the American Mathematical Society}, 306(1):1--70, 1988.

\bibitem[MR90]{marden1990on}
Al~Marden and Burt Rodin.
\newblock {\em On {T}hurston's formulation and proof of {A}ndreev's theorem}, pages 103--115.
\newblock Springer Berlin Heidelberg, Berlin, Heidelberg, 1990.

\bibitem[Pen87]{penner1987the}
Robert~C Penner.
\newblock The decorated teichm{\"u}ller space of punctured surfaces.
\newblock {\em Communications in Mathematical Physics}, 113(2):299--339, 1987.

\bibitem[Riv94]{rivin1994euclidean}
Igor Rivin.
\newblock {E}uclidean structures on simplicial surfaces and hyperbolic volume.
\newblock {\em Annals of mathematics}, 139(3):553--580, 1994.

\bibitem[RS{\etalchar{+}}87]{rodin1987convergence}
Burt Rodin, Dennis Sullivan, et~al.
\newblock The convergence of circle packings to the {R}iemann mapping.
\newblock {\em Journal of Differential Geometry}, 26(2):349--360, 1987.

\bibitem[Sch08]{schlenker2008circle}
Jean~Marc Schlenker.
\newblock Circle patterns on singular surfaces.
\newblock {\em Discrete \& Computational Geometry}, 40(1):47--102, 2008.

\bibitem[Spr08]{springborn2008variational}
Boris~A Springborn.
\newblock A variational principle for weighted {D}elaunay triangulations and hyperideal polyhedra.
\newblock {\em Journal of Differential Geometry}, 78(2):333--367, 2008.

\bibitem[Ste99]{stephenson1999approximation}
Kenneth Stephenson.
\newblock The approximation of conformal structures via circle packing.
\newblock {\em Series in Approximations and decompositions}, 11:551--582, 1999.

\bibitem[Ste05]{stephenson2005introduction}
Kenneth Stephenson.
\newblock {\em Introduction to circle packing: The theory of discrete analytic functions}.
\newblock Cambridge University Press, 2005.

\bibitem[SWGL15]{sun2015discrete}
Jian Sun, Tianqi Wu, Xianfeng Gu, and Feng Luo.
\newblock Discrete conformal deformation: {A}lgorithm and experiments.
\newblock {\em SIAM Journal on Imaging Sciences}, 8(3):1421--1456, 2015.

\bibitem[Thu79]{thurston1979geometry}
William~P Thurston.
\newblock {\em The geometry and topology of three-manifolds}.
\newblock Princeton University Princeton, NJ, 1979.

\bibitem[Ush99]{ushijima1999canonical}
Akira Ushijima.
\newblock A canonical cellular decomposition of the {T}eichm{\"u}ller space of compact surfaces with boundary.
\newblock {\em Communications in mathematical physics}, 201(2):305--326, 1999.

\bibitem[ZGZ{\etalchar{+}}14]{zhang2014unified}
Min Zhang, Ren Guo, Wei Zeng, Feng Luo, Shing-Tung Yau, and Xianfeng Gu.
\newblock The unified discrete surface {R}icci flow.
\newblock {\em Graphical Models}, 76(5):321--339, 2014.

\bibitem[Zhu19]{zhu2019discrete}
Xiang Zhu.
\newblock {\em Discrete Uniformization Theorem for Inversive Distance Circle Packings}.
\newblock PhD thesis, Tsinghua University, Beijing, China, 2019.

\end{thebibliography}

\end{document}